\documentclass{amsart}
\usepackage{amsmath,amsfonts,amssymb,amsthm}
\usepackage{hyperref}
\usepackage[latin2]{inputenc}
\usepackage[a4paper]{geometry}
\usepackage[color,matrix,arrow,all]{xy}
\usepackage{stackrel,relsize}
\usepackage{lineno}
\usepackage{relsize}
\usepackage{scalerel}
\usepackage{graphicx}

\newcommand{\comment}[2]{#2}
\newcommand{\ds}{\displaystyle\sum}

\newcommand{\complex}{\mathbb{C}}
\newcommand{\nat}{\mathbb{N}}
\newcommand{\integ}{\mathbb{Z}}
\newcommand{\orb}{\mathcal{O}}
\newcommand{\op}{\operatorname}

\newcommand{\Hom}{\operatorname{Hom}}

\newcommand{\End}{\operatorname{End}}
\newcommand{\Ext}{\operatorname{Ext}}
\newcommand{\GL}{\operatorname{GL}}
\newcommand{\SL}{\operatorname{SL}}
\newcommand{\Sp}{\operatorname{Sp}}
\newcommand{\SO}{\operatorname{SO}}
\newcommand{\spin}{\operatorname{Spin}}
\newcommand{\Gtwo}{\operatorname{G_2}}
\newcommand{\E}{\operatorname{E_6}}
\newcommand{\rep}{\operatorname{rep}}

\newcommand{\Sym}{\operatorname{Sym}}
\newcommand{\dif}{\mathcal{D}}
\newcommand{\lie}{\mathfrak{g}}
\newcommand{\under}{\underline}
\newcommand{\ove}{\overline}
\newcommand{\dual}{\mathbb{D}}

\newcommand\charC{{\operatorname{charC}}}

\comment{
\newtheorem{theorem}{Theorem}[section]
\newtheorem*{theoremnonum}{Theorem}
\newtheorem{lemma}[theorem]{Lemma}
\newtheorem{proposition}[theorem]{Proposition}
\newtheorem{corollary}[theorem]{Corollary}
\newtheorem*{varthm}{}
\newtheorem{conjecture}{Conjecture}
\theoremstyle{definition}
\newtheorem{remark}[theorem]{Remark}
\newtheorem{definition}[theorem]{Definition}
\newtheorem{notation}[theorem]{Notation}
\newtheorem{note}[theorem]{Note}
\newtheorem{example}[theorem]{Example}
\newtheorem{xca}[theorem]{Exercise}
\newtheorem{idea}[theorem]{Idea}
\newtheorem{question}[theorem]{Question}
}{}

\newtheorem{theorem}[equation]{Theorem}

\newtheorem{lemma}[equation]{Lemma}
\newtheorem{proposition}[equation]{Proposition}
\newtheorem{corollary}[equation]{Corollary}

\newtheorem*{prb}{Problem}

\theoremstyle{definition}
\newtheorem{remark}[equation]{Remark}
\newtheorem{definition}[equation]{Definition}
\newtheorem{notation}[equation]{Notation}

\newtheorem{example}[equation]{Example}

\def\presuper#1#2{\mathop{}
  \mathopen{\vphantom{#2}}^{#1}
  \kern-\scriptspace #2}

\newcommand{\calO}{{\mathcal O}}
\newcommand{\calM}{{\mathcal M}}
\newcommand{\CC}{\complex}
\def\AAA{\widehat{{\mathrm A}\hspace{-0.5ex}{\mathrm A}}}

\def\EEE{\widehat{\reflectbox{\ensuremath{\mathrm E}}\hspace{-0.4ex}{\mathrm E}}_6}

\def\calO{{\mathcal O}}

\numberwithin{equation}{section}

\begin{document}

\title{On categories of equivariant $\dif$-modules}
\author{Andr\'as C. L\H{o}rincz, Uli Walther}
\thanks{UW gratefully acknowledges NSF support through grant DMS-1401392.}
\date{}

\begin{abstract}
Let $X$ be a variety with an action by an algebraic group $G$. In this paper we discuss  various properties of $G$-equivariant $\dif$-modules on $X$, such as the decompositions of their global sections as representations of $G$ (when $G$ is reductive), and descriptions of the categories that they form. When $G$ acts on $X$ with finitely many orbits, the category of equivariant $\dif$-modules is equivalent to the category of finite-dimensional representations of a finite quiver with relations. We describe explicitly these categories for irreducible $G$-modules $X$ that are spherical varieties, and show that in such cases the quivers are almost always representation-finite (\emph{i.e.} with finitely many indecomposable representations).
\end{abstract}

\maketitle

\tableofcontents

\vspace*{5mm}

\section{Introduction}
\label{sec:intro}
A standard method for studying a complex algebraic variety $X$ is to
investigate the modules over its structure sheaf $\calO_X$,
and---in the affine, projective, or toric case---its coordinate ring. If a group
$G$ acts on $X$ then induced is an action of $G$ on $\calO_X$. The
collection of equivariant $\calO_X$-modules, those that
inherit this $G$-action, often reflects many details about the
geometry of $X$. A spectacular case is when $X$ is toric where Picard
group, Betti cohomology, and various other invariants can be computed
entirely from equivariant data.

If $X$ is smooth, then among the $\calO_X$-modules one finds a second
special category, the holonomic (left or right) $\dif_X$-modules. For
general $X$, the appropriate substitute arises from an embedding of
$X$ into a smooth $X'$; the holonomic $\dif_{X'}$-modules with
support in $X$ form a category that is
essentially independent of the embedding. 
The Riemann--Hilbert correspondence of Kashiwara and Mebkhout sets up
a bijection between the regular holonomic $\dif_X$-modules 
(for which the derived solutions exhibit polynomial growth at all
singularities) and perverse sheaves. This, and its derived version,
imply that this category encodes topological, algebraic and analytic
information on $X$. Included in this framework are de Rham cohomology,
regular and irregular connections on the smooth locus, and all
(iterated) local cohomology modules computed from $\calO_X$.

Whenever a group $G$ acts on a smooth complex variety $X$ 
then $X$ supports the category of (weakly) $G$-equivariant
$\dif$-modules $\calM$. 
For example, in the case of
a monomial action of a torus on an affine space, the
$A$-hypergeometric systems introduced by Gel'fand et
al.\ \cite{GGZ,GZK} and their generalizations from \cite{MMW05} are
of this type.
As $G$ acts on $X$, the Lie algebra elements induce vector fields on
$X$ and as such become sections of $\dif_{X}$, acting on $\calM$. For a weakly
$G$-equivariant module, differentiation 
yields a second action of the Lie algebra on $\calM$.
If these two actions by the Lie algebra on $\calM$ coincide, the corresponding
module is \emph{strongly equivariant}. There are several sources for
strong equivariance, including $\calO_X$ and all its local cohomology modules 
supported on equivariant closed subsets. These have been studied
intensively recently using methods related to ours, see
\cite{perlmana, lyubez, perlmanhyper,perlmanlyubez,claudiu2,claudiu4,claudiu3}.

Strong equivariance  is the theme of our
investigations here. In the greater part of this article (and for the
remainder of the introduction), we shall assume that $X$ is a finite
union of $G$-orbits. In this situation, strongly equivariant (coherent)
$\dif_X$-modules  
are always
regular holonomic.
To get a feeling of the constraints that (strong)
equivariance imposes, consider $\complex^*$ acting on $X=\complex^1$ by the
usual multiplication. Here, as always, $\dif_X$ is weakly (but not strongly)
equivariant. However, even on the class of holonomic modules the two
concepts differ;
the irreducible objects in
the three categories are as follows: 1) all the simple quotients of
$\dif_X$ for the holonomic category; 2) the modules $\dif_X/\dif_X x$ and
$\dif_X/\dif_X \partial_x$ as well as all modules 
$\dif_X/\dif_X(x\partial_x-\alpha)$ with $\alpha\in\complex\setminus 
\integ$ for the
weakly equivariant category, while in the strongly equivariant case 3), this reduces
to just 
$\dif_X/\dif_X x$ and $\dif_X/\dif_X \partial_x$
(and extensions in the category of strongly equivariant modules are not
arbitrary: 
the
Euler-operator must act semi-simply). The  modules $\dif_X/\dif_X x$
and $\dif_X/\dif_X \partial_x$ correspond to the
two orbits $\{0\}$ and $\complex\setminus\{0\}$ of the action.
In
quite good generality, this dialogue of orbits and strongly
equivariant $\dif_X$-modules holds true, although a further ingredient
is needed in the form of a representation of the stabilizer groups, as can be seen in \cite[Section 11.6]{htt}. In this dictionary, the simple strongly equivariant
$\dif$-modules are indexed by irreducible equivariant local systems on
the orbits. Explicit realizations of the $\dif_X$-modules in question
are in general difficult to
obtain (see Open Problem 3 in \cite[Section~6]{macvil}).  We take this
challenge as our point of departure.

One way of studying a $\dif_X$-module is via its characteristic
variety. In the presence of a $G$-action, the union $\Lambda$ of all
the conormal bundles to the orbits is the fiber over zero of the
moment map. If $G$ is semi-simple and simply connected, we prove here
that the categories of $G$-equivariant $\dif_X$-modules, and of
$\dif_X$-modules with characteristic variety inside $\Lambda$
agree. The key point is that in this case strong equivariance is
preserved under extensions. 

If $G$ is linear reductive, topologically connected, and if $X$ is a
smooth $G$-variety, then we give estimates for the multiplicities of
weights in global section modules of strongly equivariant simple
$\dif$-modules. These come in terms of multiplicities attached to the moment
map, and of the characteristic cycle. If $B$ is a Borel subgroup of
$G$ and if $X$ happens to have an open $B$-orbit ($X$ is then \emph{spherical}),
we show that these multiplicities are either zero or
one (the global sections module is \emph{multiplicity-free}, as well as (in most cases) its characteristic cycle), Theorem \ref{thm:multfreee}.

Various categories of holonomic $\dif$-modules (or perverse sheaves) often admit interpretations as categories of representations of quivers \cite{gel-mac-vil,vilon}, but explicit descriptions of these quivers are not easy to obtain \cite{bragri,galligo-granger-1,galligo-granger-2,perlmana,binary, mac-vil-cusp,nang, perlmanhyper}.
The largest part of this article is concerned with an important case of such type, a spherical vector space $X$ where $G$ is linear reductive and
connected.  In \cite{kac} is given a classification of all group
actions that fit this setup, by way of a finite list of families plus
some isolated cases. In Section \ref{sec:examples}, we examine these cases one by one
and determine in each case the (finite) quiver with relations that
encodes the category of strongly $G$-equivariant
$\dif_X$-modules. There are strong general constraints on the quivers that can
show up this way. For example,
up to relations, the number of paths between any two
vertices in these
quivers is at most one by Corollary \ref{cor:quivnice}.

This atlas of quivers to equivariant
$\dif$-modules on spherical vector spaces is created explicitly by
identifying the simple modules within each category, as well as  their
projective covers. The cases with
or without semi-invariant (\emph{i.e.}, whether the complement of the
big orbit is a divisor) behave rather differently; heuristically, the
former case exhibits more interesting structure while the latter is
nearer to the semi-simple case. The other source of interesting
quivers is the existence of non-trivial local systems on orbits,
caused by non-connected stabilizers, which appear in several
of these families.

Our investigations show that the basic building block for these
quivers is a ``doubled $A_n$-quiver with relations'' $\AAA_n$;
shown in \eqref{eq:double}. All but one of our quivers are
finite disjoint unions of such $\AAA_n$ (allowing for $n=1$,
an isolated vertex). The one exception arises from the standard action
of $\Sp_4\otimes \GL_4$ on the $4\times 4$ matrices. In this case the
quiver has a vertex of in- and out-degree $3$; in analogy to
$\AAA_n$ we call this the ``doubled $E_6$'', denoted $\EEE$. We show
in Theorem \ref{thm:EE} that $\EEE$ is of (domestic) tame representation
type, and in Theorem \ref{thm:AA} that all  $\AAA_n$ are of finite representation type.

\medskip

Here is, in brief, the outline for the paper. In Section \ref{sec:prelim}  we collect
basic material on equivariance and on quivers. In Section \ref{sec:finite}, we
consider the general case of finitely many  orbits, and consider
multiplicities and the moment map. We zoom into the spherical case and
discuss projective covers. In Section \ref{sec:tech}
we collect various tools to be used in Section \ref{sec:examples}: a reduction
technique for subgroups,
Bernstein--Sato polynomials, and the Fourier transform. In Section \ref{sec:examples}, we then study the category of strongly equivariant
$\dif_X$-modules on spherical vector spaces. In the last section, we give some concluding remarks regarding explicit presentations of equivariant $\dif$-modules and descriptions of their characteristic cycles. 

\medskip

The quiver-theoretic description of the category of equivariant $\dif$-modules on the space of generic matrices in Theorem \ref{thm:gener} is an important ingredient in the article \cite{lyubez}, where the explicit $\dif$-module structures of the (iterated) local cohomology modules supported on determinantal varieties are determined. In particular, the Lyubeznik numbers of determinantal rings are computed. This is done for Pfaffian rings in \cite{perlmanlyubez}. Similar computations could be pursued based on these methods for the other spherical vector spaces that are considered in Section \ref{sec:examples}.

\section{Preliminaries}
\label{sec:prelim}

Throughout we work over the field of complex numbers $\complex$. Unless otherwise stated, we assume that $X$ is a connected smooth complex algebraic variety equipped with the algebraic action of a connected linear algebraic group $G$. 

\subsection{Equivariant $\dif$-modules} \label{subsec:equiv}

Let $\dif_X$ be the sheaf of differential operators on $X$ and $\lie$ the Lie algebra of $G$. Differentiating the action of $G$ on $X$ yields a map $\lie \to \Gamma(X,\dif_X)$. This map can be extended to an algebra map $U(\lie) \to \Gamma(X,\dif_X)$, where $U(\lie)$ denotes the universal enveloping algebra of $\lie$.

A $\dif_X$-module $M$ on $X$ is a quasi-coherent sheaf of (left) $\dif_X$-modules. We call $M$ a (strongly) $G$-\textit{equivariant} $\dif$-module, if we have a $\dif_{G\times X}$-isomorphism
\[
\tau\colon p^*M \rightarrow m^*M,
\]
where $p$ and $m$ are the projection and multiplication maps
\[
p\colon G\times X\to X,\qquad\qquad m\colon G\times X\to X
\]
respectively, and $\tau$ satisfies the usual compatibility conditions on $G\times G\times X$ (see \cite[Definition 11.5.2]{htt}). Roughly speaking, this amounts to $M$ admitting an algebraic $G$-action such that differentiating it coincides with the action induced from $\lie \to \Gamma(X,\dif_X)$ (see \cite[Proposition 2.6]{vander2}). A $\dif$-module morphism between $G$-equivariant $\dif_X$-modules automatically preserves $G$-equivariance (since $G$ is connected, see \cite[Proposition 3.1.2]{vander}).

We denote the category of quasi-coherent (resp.\ coherent) $\dif$-modules by  $\op{Mod}(\dif_X)$ (resp.\ $\op{mod}(\dif_X)$), and the full subcategory of quasi-coherent (resp.\ coherent) equivariant $\dif$-modules by $\op{Mod}_G(\dif_X)$ (resp.\  $\op{mod}_G(\dif_X)$). These are Abelian categories that are stable under taking subquotients within $\op{Mod}(\dif_X)$ (resp.\ $\op{mod}(\dif_X)$), see \cite[Proposition 3.1.2]{vander}. In particular, if a map $\tau$ as above exists, it must be unique, and equivariance of a $\dif$-module should be thought of as a condition, rather then additional data.

For the results where we allow $X$ to be not necessarily smooth, \textit{we will assume at least that we can find an equivariant closed embedding $X\hookrightarrow X'$ with $X'$ a smooth $G$-variety}. For example, this is possible whenever $X$ is quasi-affine (see \cite[Theorems 1.5 and 1.6]{popvin}), or when $X$ is a normal quasi-projective $G$-variety (see \cite[Theorem 1]{sumi}). Then we define $\op{Mod}(\dif_X):= \op{Mod}^X(\dif_{X'})$, where $\op{Mod}^X(\dif_{X'})$ is the full subcategory of $\dif_{X'}$-modules supported on $X$. Similarly, one can define  $\op{mod}(\dif_X), \op{Mod}_G(\dif_X), \op{mod}_G(\dif_X)$ and one checks that the definition is independent of the embedding. We note that one could also define these categories whenever $X$ can be embedded locally into smooth varieties in an equivariant way (for example, whenever $X$ is a normal $G$-variety (see \cite[Lemma 8]{sumi})). Several results can be easily extended to this case.

We call a (possibly infinite-dimensional) vector space $V$ a \textit{rational} $G-$\textit{module}, if $V$ is equipped with a linear action of $G$, such that every $v\in V$ is contained in a finite-dimensional $G$-stable subspace on which $G$ acts algebraically.

For an equivariant $\dif_X$-module $M$ on a smooth $G$-variety $X$, the cohomology groups $H^i(X,M)$ are rational $G$-modules for any $i\in \nat$. Moreover, the same is true for all local cohomology modules $H^k_Z(X,M)$ supported on a closed $G$-stable subset $Z$.

If $f$ is any $G$-equivariant map $f:X \to Y$ between smooth $G$-varieties $X,Y$, then the $\dif$-module-theoretic direct image $f_+$ and inverse image $f^*$  (and their derived functors) send (complexes of) equivariant modules to equivariant modules. We note that in the case when $f$ is an open embedding, $f_+$ coincides with the $\calO_X$-theoretic $f_*$ 

For a finite-dimensional rational representation $V$ of $G$, we can define the $\dif$-module 
\[
P(V):=\dif_X\otimes_{\under{U}(\lie)} \under{V} \in \op{mod}(\dif_X),
\]
where underline in $\under{U}(\lie),\under{V}$ indicates constant sheaves. Though we consider it only for finite-dimensional representations $V$, the construction of $P(V)$ is similar to that of the $\dif_X$-modules on flag varieties used in proving the Kazhdan--Lusztig conjecture \cite{localisation, brykash}. 

The following result is most likely known by experts, but we provide the proof for sake of completeness.

\begin{lemma}\label{lem:rep}
Let $V$ be a finite-dimensional $G$-module and $M\in \op{Mod}_G(\dif_X)$.  Then:
\begin{itemize}
\item[(a)] $\Hom_{\dif_X}(P(V), M)\cong \Hom_G (V,\Gamma(X,M)),$
\item[(b)] If $j: U \to X$ is an open embedding, then $j^*(P(V)) = \dif_U\otimes_{\under{U}(\lie)} \under{V},$
\item[(c)] $P(V)\in  \op{mod}_G(\dif_X)$.
\end{itemize}
\end{lemma}

\begin{proof}
Part (a) follows by adjunction:
$$\Hom_{\dif_X} (P(V),M)\cong \Hom_{\under{U}(\lie)}(\under{V},\mathcal{H}om_{\dif_X}(\dif_X,M))\cong\Hom_{U(\lie)}(V,\Gamma(M))\cong \Hom_G(V,\Gamma(M)).$$ 
For part (b), take an arbitrary $N \in \op{Mod}(\dif_U)$. Again by adjunction we have:
\[\Hom_{\dif_U}(j^*(P(V)), N) \cong \Hom_{\dif_X}(P(V),j_* N) \cong \Hom_{U(\lie)}(V, \Gamma(U,N)).\]
Together with part (a), this shows that the $\dif_U$-modules $\dif_U\otimes_{\under{U}(\lie)} \under{V}$ and $j^*(P(V))$ represent the same functor, hence they are isomorphic.

For part (c), we first show that $P(V)$ is a weakly $G$-equivariant $\dif$-module, in the sense of \cite[Section 2]{vander2}. We use \cite[Proposition 2.2]{vander2}. Let $R$ be a $k$-algebra and denote $X_R = \op{Spec} R \times X$. An $R/k$-point $i_g : \op{Spec} R \to G$ induces an $R$-automorphism $g: X_R \to X_R$. As in part (b), for an affine open $U\subset X_R$, we have $\Gamma(U, \dif_{X_R} \otimes_{\under{U}(\lie)}\under{V}) = \dif_{X_R}(U) \otimes_{U(\lie)}V$.
 Analogous to the discussion following \cite[Proposition 2.2]{vander}, it is enough to show that we have isomorphisms
\[r_g:  \dif_{X_R}(U) \otimes_{U(\lie)} V \to \dif_{X_R}(g^{-1}U) \otimes_{U(\lie)}V,\]
satisfying $r_{gh}=r_h r_g$ and $r_{1} = \op{id}$.  For $D\in \dif_{X_R}(U)$ let $g^* D \in \dif_{X_R}(g^{-1}U)$ be the operator that for any $f\in \calO_{X_R}(g^{-1}U)$ yields $(g^* D) \cdot f =  (D \cdot (f \circ g^{-1})) \circ g  \in \calO_{X_R}(g^{-1}U)$. Then we define $r_g (D \otimes v) = g^* D \otimes g^{-1} v$, and this satisfies the requirements.

Now we show that $P(V)\in  \op{mod}_G(\dif_X)$. We have two actions of $\lie$ on $P(V)$, one induced by the weakly equivariant structure (the tensor representation given by $r_g$) and the other induced by the $\dif_X$-module structure via the map $\psi: 
\underline\lie \to \dif_X$. According to \cite[Proposition 2.6]{vander2}, we must show that these actions coincide.

Let $U\subset X$ be an affine open, $D\in \dif_X(U)$, $v\in V$ and $\xi \in \lie$. Then the weakly equivariant action of $\lie$ on $P(V)$ is
\begin{equation}\label{eq:action1}
\xi \cdot (D \otimes v) = (\xi \cdot D) \otimes v + D \otimes \xi v = [\psi(\xi),D]\otimes v + D \otimes \xi v.
\end{equation}
On the other hand, the action of $\lie$ via the $\dif_X$-module structure gives
\begin{equation}\label{eq:action2}
\psi(\xi) \cdot (D \otimes v) = \psi(\xi) D \otimes v = [\psi(\xi),D]\otimes v + D\psi(\xi)\otimes v = [\psi(\xi),D]\otimes v + D \otimes \xi v.
\end{equation}
Hence the actions (\ref{eq:action1}) and (\ref{eq:action2}) coincide, thus finishing the proof.
\end{proof}

\begin{remark}\label{rem:notequi}
Naturally, we could define $P(V)$ for any finite-dimensional $\mathfrak{g}$-module. For example, when $X$ is affine, this yields a $\dif_X$-module with a locally finite $\mathfrak{g}$-action (with the $\lie$-action induced by the map map $\lie \to \dif_X$), i.e. every $x\in P(V)$ is contained in a finite-dimensional $\lie$-stable subspace of $P(V)$. Several results of this paper could be extended readily for such $\dif_X$-modules, but we focus mainly on $G$-equivariant $\dif_X$-modules.
\end{remark}

When $G$ is a connected linear reductive group, the category of rational (possibly infinite-dimensional) representations of $G$ is semi-simple, and the irreducible (finite-dimensional) representations $V_\lambda$ are in 1-to-1 correspondence with (integral) dominant weights $\lambda$. Let $\Pi$ denote the set of all dominant weights of $G$.
For a highest weight $\lambda\in \Pi$ we put $P(\lambda):=P(V_\lambda)$. We note that the $\dif_X$-module $P(\lambda)$ has an explicit presentation. Namely, pick a highest weight vector $v_\lambda$ of $V_\lambda$. Then we can write a $U(\lie)$-isomorphism $V_\lambda\cong U(\lie)/\!\op{Ann}_{U(\lie)}v_\lambda$, where $\op{Ann}_{U(\lie)}v_\lambda$ denotes the annihilator of $v_\lambda$ in $U(\lie)$ (the generators of $\op{Ann}_{U(\lie)}v_\lambda$ are well-known: see \cite[Theorem 21.4]{humph}). Hence we have a $\dif_X$-module isomorphism
\begin{equation}\label{eq:highest}
P(\lambda) \cong \dif_X \otimes_{U(\lie)} (U(\lie)/\!\op{Ann}_{U(\lie)}v_\lambda) \cong \dif_X /(\op{Ann}_{U(\lie)}v_\lambda).
\end{equation}

\begin{definition}\label{def:multfree}
We call a rational $G$-module $V$ \textit{multiplicity-finite} if we have an isotypical decomposition
\[
\qquad\qquad\qquad\qquad V=\bigoplus_{\lambda \in \Pi} V_{(\lambda)}, \qquad\text{ where }
V_{(\lambda)}\cong V_\lambda^{\oplus m_\lambda(V)} \text{ and }
m_\lambda(V) \in \nat \text{ for all } \lambda\in\Pi.
\]
We call $V$ \textit{multiplicity-bounded} if the set $\{m_\lambda(V)\}_{\lambda\in \Pi}$ is bounded, and \textit{multiplicity-free} if this bound is $1$, \emph{i.e.} $m_\lambda(V)\leq 1$ for all $\lambda\in \Pi$.
\end{definition}

A smooth variety $X$ is $\dif$-\textit{affine} if the global sections functor 
\[
\Gamma\colon \op{Mod}{\dif_X} \longrightarrow \op{Mod} {\Gamma(X,\dif_X)}
\]
is exact and faithful on objects. In this case $\Gamma$ induces an equivalence of categories (preserving equivariance), and we shall freely identify coherent $\dif_X$-modules with their global sections. Smooth affine varieties, projective spaces and (partial) flag varieties are $\dif$-affine (see \cite[Theorem 1.6.5 and Corollary 11.2.6]{htt}).

\begin{proposition}\label{prop:proj}
Assume that $X$ is $\dif$-affine and $G$ is linearly reductive. Let $V$ be a finite-dimensional $G$-module. Then $P(V)$ is projective in $\op{Mod}_G(\dif_X)$ (resp.\ $\,\op{mod}_G(\dif_X)\,$), and the category $\op{Mod}_G(\dif_X)$ (resp.\ $\,\op{mod}_G(\dif_X)\,$) has enough projectives.
\end{proposition}

\begin{proof}
Since $G$ is reductive and $\Gamma(X,-)$ exact, $P(V)$ is projective
by Lemma \ref{lem:rep} (a). Let $M$ be in $\op{mod}_G(\dif_X)$. Since
$M$ is coherent, we can take $x_1,x_2,\dots x_l\in \Gamma(X,M)$ that
generate $M$ everywhere. The action of $G$ on $M$ is locally finite, so we can find a finite-dimensional representation $V$ of $G$ that contains $x_1,\dots,x_l$. Then we have a surjective map $P(V)\twoheadrightarrow M$. The claim for $\op{Mod}_G(\dif_X)$ is analogous.
\end{proof}

Note that $\op{Mod}_G(\dif_X)$ always has enough injectives (see \cite[Lemma 1.5.3 and Corollary 2.8]{vander2}).

The following lemma will be useful in determining irreducible or indecomposable $\dif_X$-modules:

\begin{lemma}\label{lem:irrind}
Let $X$ be $\dif$-affine and $M$ an equivariant $\dif_X$-module. Assume that $M$ is globally generated by a highest weight vector $v_\lambda \in \Gamma(X,M)$ of weight $\lambda \in \Pi$, and $m_\lambda(\Gamma(X,M))=1$. Then $M$ has a unique (non-zero) irreducible quotient and $\End_{\dif}(M)=\complex$.
\end{lemma}

\begin{proof}
We identify $\dif_X$-modules with their modules of global sections. If $N$ is a proper $\dif$-submodule of $M$, then we must have $m_\lambda(N)=0$. Let $M'$ denote the sum of all proper submodules of $M$. Then $m_\lambda(M')=0$, hence $M/M'$ is the required irreducible $\dif$-module. 

Any $\dif$-module map $f:M\to M$ induces a $G$-module map on global sections. By Schur's lemma we have $f(v_\lambda)=c\cdot v_\lambda$, for some $c\in \complex$. Since $M$ is generated by $v_\lambda$ as a $\dif$-module, we obtain $f=c \cdot \op{id}.$
\end{proof}

Although $\op{Mod}_G(\dif_X)$ in general is not is a Serre subcategory of $\op{Mod}(\dif_X)$, it is one when $G$ is  semi-simple (\emph{i.e.} when $\lie$ is a semi-simple Lie algebra):

\begin{proposition}\label{prop:serre}
Assume that $G$ is semi-simple, and $X$ is a (not necessarily smooth) $G$-variety. Then $\op{Mod}_G(\dif_X)$ (resp.\ $\op{mod}_G(\dif_X)$)  is closed under extensions in $\op{Mod}(\dif_X)$ (resp.\ $\op{mod}(\dif_X)$).
\end{proposition}

\begin{proof}
By taking an equivariant closed embedding to a smooth $G$-variety, we may assume that $X$ is itself smooth. Take an exact sequence 
\begin{gather}\label{eqn-MQN}
  0\to M \to Q \to N\to 0,
\end{gather}
where $M,N \in \op{Mod}_G(\dif_X)$, and $Q \in \op{Mod}(\dif_X)$. We want to show that $Q \in \op{Mod}_G(\dif_X)$.

First, assume that $X$ is $\dif$-affine. As $\lie$-modules, we can
write $M=\bigoplus_{i\in I} M_i$ and $N=\bigoplus_{j\in J} N_j$, where
$M_i$ and $N_j$ are finite-dimensional simple $\lie$-modules. The
semi-simple Lie algebra $\lie$ acts on $Q$ in the usual way via
$U(\lie)\to \dif_X$, and this map is compatible with the morphisms in \eqref{eqn-MQN}. In order to integrate the action to the group $G$, it is enough to show that the sequence above splits as $\lie$-modules, \emph{i.e.} that $Q$ is a semi-simple $\lie$-module with finite-dimensional summands. This follows if we show that the action of $\lie$ on $Q$ is locally finite.

We show that the action on $\tilde{x}$ is locally finite, where $\tilde{x}$ is a lift of an arbitrary element $x$ in $N$.  We can find a finite-dimensional representation $V\subseteq N$ containing $x$ with a basis, say, $n_1=x, n_2, \dots , n_k$. Take lifts of the basis elements $\tilde{n}_j$ in $Q$ (with $\tilde{x}=\tilde{n}_1$), and denote their span by $\tilde{V}$. Take a basis $\xi_1,\dots, \xi_l$ for the Lie algebra $\lie$. Then for all $i=1,\dots, l$, $j=1\dots k$, we can write
$$\xi_i \tilde{n}_j = a^{ij}_1 \tilde{n}_1 + \dots a^{ij}_k \tilde n_{k}+ m_{ij},$$
where $a^{ij}_1,\dots,a^{ij}_k$ are scalars and $m_{ij}\in M$. Since $M$ is a locally finite $\lie$-module, for each $m_{ij}$ we can pick a finite-dimensional representation $M_{ij}$ containing it. Then the space
$$\tilde{V} + \sum_{i=1}^{l} \sum_{j=1}^{k} M_{ij}$$
is a finite-dimensional $\lie$-representation containing $\tilde{x}$, hence the desired conclusion follows in this case.

Now assume that $X$ is a $G$-stable open subset of a projective space
$X'=\mathbb{P}(V)$, where $V$ is a finite-dimensional rational $G$-module, and let $j:X\to X'$ denote the embedding. Applying $j_*$ we get an exact sequence 
$$0\to j_*M \to j_*Q \to j_*N,$$
with $j_*M,j_*N \in \op{Mod}_G(\dif_{X'})$. Since $\op{Mod}_G(\dif_{X'})$ is closed under taking submodules, we can write an exact sequence
$$0\to j_*M \to j_*Q \to N'\to 0,$$
for some $N'\in \op{Mod}_G(\dif_{X'})$. Since the projective space $X'$ is $\dif$-affine  (see \cite[Theorem 1.6.5]{htt}), the previous argument implies $j_*Q\in\op{Mod}_G(\dif_{X'})$.  Then $j^*j_*Q = Q \in \op{Mod}_G(\dif_X)$.

Now assume that $X$ is a quasi-projective $G$-variety. By \cite[Theorem 1]{sumi} we have a $G$-equivariant embedding of $X$ into a projective space $X'$. Then $S=\ove{X}\setminus X$ is a closed $G$-stable subset of $X'$. Let $U=X'\setminus S$ and $i:X \to U$ the closed embedding.  By Kashiwara's equivalence (see \cite[Section 1.6]{htt}), $i_+$ induces an equivalence of categories between $\op{Mod}_G(\dif_X)$ and the full subcategory $\op{Mod}_G^X(\dif_U)$ of equivariant $\dif_U$-modules whose support is contained in $X$. By the previous argument, $i_+ Q \in \op{Mod}_G(\dif_U)$, hence $Q\in \op{Mod}_G(\dif_X)$.

Now we consider the general case. Since $X$ is smooth, by (\cite[Lemma
  8]{sumi}) we can cover $X$ with $G$-stable quasi-projective open
subsets $\{U_i\}_{i\in I}$. The previous argument implies that $Q_{|
  U_i}$ is a $G$-equivariant $\dif_{U_i}$-module, for all $i\in
I$. Hence, for each $i\in I$ we have an isomorphism $\tau_i : p^*(Q_{|
  U_i}) \rightarrow m^*(Q_{| U_i})$. Since $G$-equivariant structures
are unique, induced by $U(\lie)\to\dif_X$, the maps $\tau_i$ and $\tau_j$ must coincide on the the intersection $U_i \cap U_j$. Hence we can glue them to get an isomorphism $\tau: p^*Q \rightarrow m^* Q$.
\end{proof}

\subsection{Quivers}\label{subsec:quiv}

We establish some notation and review some basic results regarding the representation theory of quivers, following \cite{elements}.  A \textit{quiver} $Q$ is an oriented graph, \emph{i.e.} a pair $Q=(Q_0,Q_1)$ formed by a finite set of vertices $Q_0$ and a finite set of arrows $Q_1$. An arrow $a\in Q_1$ has a head (or target) $ha$ and a tail (or source) $ta$ which are elements of $Q_0$:
\[\xymatrix{
ta \ar[r]^{a} & ha
}\]
The complex vector space with basis given by the (directed) paths in $Q$ has a natural multiplication induced by concatenation of paths. The corresponding $\CC$-algebra is called the \textit{path algebra} of the quiver $Q$ and is denoted $\CC Q$. 

A \textit{relation} in $Q$ is a $\CC$-linear combination of paths of length at least two having the same source and target. We define a \textit{quiver with relations} $\widehat{Q}:=(Q,I)$ to be a quiver $Q$ together with a finite set of relations $I$. The \textit{quiver algebra} of $\widehat{Q}$ is the quotient $\CC Q/\langle I \rangle$ of the path algebra by the ideal generated by the relations. We always assume that the ideal of relations $\langle I\rangle$ contains any path whose length is large enough, so that the corresponding quiver algebra is finite-dimensional (see \cite[Section II.2]{elements}). We will often use the word quiver to refer to a quiver with relations.

A \textit{(finite-dimensional) representation} $V$ of a quiver $\widehat{Q}$ is a family of (finite-dimensional) vector spaces $\{V_x\,|\, x\in Q_0\}$ together with linear maps $\{V(a) : V_{ta}\to V_{ha}\, | \, a\in Q_1\}$ satisfying the relations induced by the elements of $I$. A morphism $\phi:V\to V'$ of two representations $V,V'$ of $\widehat{Q}$ is a collection of linear maps  $\phi = \{\phi(x) : V_x \to V'_x\,| \,x\in Q_0\}$, such that for each $a\in Q_1$ we have $\phi(ha)\circ V(a)=V'(a)\circ \phi(ta)$. We note that the data of a representation of $\widehat{Q}$ is equivalent to that of a module over the quiver algebra, and in fact, the category $\rep(\widehat{Q})$ of finite-dimensional representations of $\widehat{Q}$ is equivalent to that of the finitely generated $\CC Q/\langle I \rangle$-modules \cite[Section~III.1, Thm.~1.6]{elements}. This category is Abelian, Artinian, Noetherian, has enough projectives and injectives, and contains only finitely many simple objects, seen as follows.

The (isomorphism classes of) simple objects in $\rep(\widehat{Q})$ are in bijection with the vertices of $Q$. For each $x\in Q_0$, the corresponding simple $S^x$ is the representation with $(S^x)_x=\CC,\ (S^x)_y=0\mbox{ for all }y\in Q_0\setminus\{x\}$.
A (non-zero) representation of $\widehat{Q}$ is called \textit{indecomposable} if it is not isomorphic to a direct sum of two non-zero representations. For each $x\in Q_0$, we let $P^x$ (resp.\ $I^x$) denote the \textit{projective cover} (resp.\ \textit{injective envelope}) of $S^x$, as constructed in \cite[Section III.2]{elements}. In particular, for $y\in Q_0$, the dimension of $(P^x)_y$ (resp.\ $(I^x)_y$) is given by the number of paths from $x$ to $y$ (resp.\ from $y$ to $x$), considered up to the relations in $I$.

A quiver $\widehat{Q}$ is said to be of \textit{finite representation type} if it has finitely many indecomposable representations (up to isomorphism). The quiver $\widehat{Q}$ is of \textit{tame representation type} if all but a finite number of indecomposable representations of $\widehat{Q}$ of a given dimension belong to finitely many one-parameter families \cite[XIX.3, Definition 3.3]{sim-sko3}, and it is of \textit{wild representation type} otherwise. When the number of one-parameter families of indecomposables is bounded as the dimension of the representations grows, we say that $\widehat{Q}$ is of \textit{domestic tame representation type} \cite[XIX.3, Definitions 3.6, 3.10 and Theorem 3.12]{sim-sko3}.

Given two quivers with relations, $\widehat{Q} = (Q,I)$ and $\widehat{Q}' = (Q',I')$, we say that $\widehat{Q}'$ is a \textit{subquiver} of $\widehat{Q}$, if $Q'$ is a (oriented) subgraph of $Q$, and the relations generated by $I'$ contain the relations from $I$ that are induced from $Q$ to $Q'$. Clearly, this gives a surjective map $\CC Q/\langle I \rangle \to \CC Q'/\langle I' \rangle$  and $\rep(\widehat{Q}')$ is naturally a full subcategory of $\rep(\widehat{Q}).$

Any finite-dimensional $\CC$-algebra is Morita-equivalent to the quiver algebra $\CC Q/\langle I \rangle$ of some quiver with relations $\widehat{Q}=(Q,I)$ (e.g. see \cite[Theorem 3.7]{elements}). This means that the category of finitely generated modules over an algebra is equivalent to the category $\rep(\widehat{Q})$ of finite-dimensional representations of a quiver $\widehat{Q}$.

We introduce the following quiver $\AAA_n$ (for $n\geq 1$), which will be of special importance later (see Section \ref{sec:examples}):

\begin{equation}\label{eq:double}
\AAA_n: \,\,\,  \xymatrix{
(1) \ar@<0.5ex>[r]^{\alpha_1} & \ar@<0.5ex>[l]^{\beta_1} (2)
\ar@<0.5ex>[r]^{\alpha_2} & \ar@<0.5ex>[l]^{\beta_2} \dots
\ar@<0.5ex>[r]^{\alpha_{n-2}} &\ar@<0.5ex>[l]^{\beta_{n-2}} (n-1)
\ar@<0.5ex>[r]^-{\alpha_{n-1}} & \ar@<0.5ex>[l]^-{\beta_{n-1}} (n)},
\end{equation}
with all $2$-cycles zero (\emph{i.e.} relations $\alpha_i \beta_i = 0 = \beta_i \alpha_i$ for all $i=1,\dots,n-1$). By convention, $\AAA_1$ is just a vertex.

We introduce the following indecomposable representations of $\AAA_n$. Given an interval $[i,j]$ with $1\leq i\leq j \leq n$, we choose a binary $(j-i)$-tuple $\Sigma$ of the signs $+$ or $-$. Then we define the representation $I_{i,j}^\Sigma$ of $\AAA_n$ by putting $\complex$ to each vertex $(k)$ with $i\leq k \leq j$, and $0$ at the other vertices, together with the linear maps chosen as follows: if $\Sigma_l = +$ \,(resp.\ $\Sigma_l = -$), for some $1\leq l \leq j-i$, then we choose the map on the arrow $(i+l-1)\rightarrow (i+l)$ to be the identity map (resp.\ $0$), and the map on the arrow $(i+l-1)\leftarrow (i+l)$ to be $0$ (resp.\ the identity map).

\begin{example}
With the above notation, we have the following indecomposable representation $I_{2,5}^{+-+}$ of $\AAA_6$:
\[I_{2,5}^{+-+}: \vcenter{\xymatrix{
0 \ar@<0.5ex>[r]^0 & \ar@<0.5ex>[l]^0 \CC \ar@<0.5ex>[r]^1 & \CC  \ar@<0.5ex>[l]^0 \ar@<0.5ex>[r]^0 & \ar@<0.5ex>[l]^1 \ar@<0.5ex>[r]^1 \CC & \ar@<0.5ex>[l]^0 \CC \ar@<0.5ex>[r]^0 & \ar@<0.5ex>[l]^0 0
}}\]
\end{example}

\begin{theorem}\label{thm:AA}
All indecomposable representations of $\AAA_n$ are (up to isomorphism) of the form $I_{i,j}^\Sigma\,$, for some $1\leq i \leq j \leq n$ and $(j-i)$-tuple of signs $\Sigma$. In particular, the quiver $\AAA_n$ is of finite representation type.
\end{theorem}

\begin{proof}
The quiver $\AAA_n$ is a string algebra as defined in \cite[Section 3]{butring}. Hence, the indecomposables are given by either string modules or band modules (see \cite[Section 3]{butring}). The strings of $\AAA_n$ correspond precisely to the tuples $\Sigma$, and the respective string modules are the modules of the form $I_{i,j}^\Sigma\,$. Note that there are no cyclic strings, hence there are no band modules.
\end{proof}

Another quiver that we use in Section \ref{sec:examples} is a subquiver of the $\AAA_3$ quiver, defined as
\begin{equation}\label{eq:doubb}
\AAA_3^c : \xymatrix{(1) \ar@<0.5ex>[r] & \ar@<0.5ex>[l] (2)
\ar@<0.5ex>[r] & \ar@<0.5ex>[l] (3)}
\end{equation}
with all compositions of arrows zero. Clearly, the indecomposables of $\AAA_3^c$ are of the form $I_{i,j}^\Sigma$ by requiring additionally that $\Sigma$ is not equal to $++$ or $--$.

Another quiver that will appear in Section \ref{sec:examples} is the following:

\begin{equation}\label{eq:e6}
\EEE: \,\,\ \vcenter{\xymatrix{
& & (6) \ar@<0.5ex>[d]^{\alpha} & & \\
(1) \ar@<0.5ex>[r] & \ar@<0.5ex>[l] (2) \ar@<0.5ex>[r] & \ar@<0.5ex>[u]^{\beta} \ar@<0.5ex>[l] \ar@<0.5ex>[r] (3) & \ar@<0.5ex>[l] (4) \ar@<0.5ex>[r] & \ar@<0.5ex>[l] (5)
}}
\end{equation}
with all 2-cycles zero, and all compositions with the arrows $\alpha$ or $\beta$ equal to zero.

\begin{theorem}\label{thm:EE}
The quiver $\EEE$ is of domestic tame representation type.
\end{theorem}

\begin{proof}
Let $V$ be an indecomposable representation of $\EEE$. If $V(\alpha)=V(\beta)=0$, then $V$ is supported on a quiver of type $\AAA_5$, which is of finite representation type by Theorem \ref{thm:AA}. 

Next, assume that $V(\beta) \neq 0$ and $V(\alpha) = 0$. Disregarding vertex $(6)$, by Theorem \ref{thm:AA} we can decompose the $\AAA_5$ part of $V$ into indecomposables $I_{i,j}^\Sigma$, with $i\leq 3 \leq j$ (since $V$ is indecomposable). If among these there is an indecomposable $I_{i,j}^\Sigma$ such that a map to vertex $(3)$ is non-zero, then we extend the representation $I_{i,j}^\Sigma$ to a representation of $\EEE$ by placing the zero space at vertex $(6)$. By construction, this representation is a summand of $V$, which is a contradiction, since $V$ is indecomposable with $V(\beta) \neq 0$. This shows that all maps of $V$ pointing to $(3)$ are zero. Furthermore, the vertices $(1)$ and $(5)$ are \textit{nodes}, since the paths of length $2$ passing through them are zero (see \cite[pg. 12]{binary}). By splitting these nodes as in \cite[Lemma 2.7]{binary}, we conclude that $V$ can be viewed as a representation of the following quiver
\begin{equation}\label{eq:blow}
\widehat{B}_8:\, \, \,\vcenter{\xymatrix{
(7) \ar[dr] & & (6)  & & (8) \ar[dl] \\
(1) & (2) \ar[l] & (3) \ar[l] \ar[u] \ar[r] & (4) \ar[r] & (5)
}}
\end{equation}
with the compositions $(7)\to (2) \to (1)$ and $(8) \to (4) \to (5)$ zero. 

Dually, if we assume that $V(\alpha)\neq 0$ and $V(\beta)=0$, then $V$ can be realized as a representation of a quiver of type $\widehat{B}_8$ as above, but with all arrows reversed, which we denote by $\widehat{B}_8^o$.

Now we show that if $V$ is any indecomposable representation of $\EEE$ then we cannot have both $V(\alpha) \neq 0$ and $V(\beta)\neq 0$. Assume without loss of generality that we have $V(\alpha)\neq 0$. As vertex $(6)$ of $\EEE$ is a node, we can split it as in \cite[Lemma 2.7]{binary}, and view $V$ as a representation of the quiver 
\begin{equation}\label{eq:blow6}
\vcenter{\xymatrix{
& (6) &  & (6') \ar[dl]^{\alpha} & \\
(1) \ar@<0.5ex>[r] & \ar@<0.5ex>[l] (2) \ar@<0.5ex>[r] & \ar@<0.5ex>[ul]^{\beta} \ar@<0.5ex>[l] \ar@<0.5ex>[r] (3) & \ar@<0.5ex>[l] (4) \ar@<0.5ex>[r] & \ar@<0.5ex>[l] (5)
}}\end{equation}
Disregarding vertex $(6)$ of the quiver \eqref{eq:blow6}, we decompose the corresponding part of $V$ into indecomposables. Since $V(\alpha)\neq 0$, there is an indecomposable $I$ with $I(\alpha)\neq 0$ coming from the quiver $\widehat{B}_8^o$, as discussed above. Clearly, for any such indecomposable $I$ the sum of all maps to vertex $(3)$ is surjective (otherwise it would have a summand isomorphic to the simple representation at vertex $(3)$). Since all maps composed with $\beta$ are zero, we can extend $I$ to the quiver \eqref{eq:blow6} by placing the zero space at vertex $6$. The obtained representation is a summand of $V$, showing that $V(\beta)=0$.

We have shown  that if $V$ is any indecomposable of $\EEE$, then it comes from an indecomposable supported on either $\AAA_5$, $\widehat{B}_8$ or $\widehat{B}_8^o$. Hence, we are left to show that $\widehat{B}_8$ is of domestic tame representation type.

For this, we consider the \textit{Tits form} $\hat{q}$ of  $\widehat{B}_8$ as   in \cite{tits}:
\[
\hat{q}(x)=\ds_{i=1}^8 x_i^2 - x_1 x_2 - x_2 x_3 -x_3 x_4 - x_4 x_5 - x_3 x_6 - x_2 x_7 - x_4 x_8 + x_1 x_7 + x_5 x_8.
\]
By an elementary consideration (e.g. computing eigenvalues), we see that $\hat{q}$ is a positive semi-definite quadratic form, that is $\hat{q}(x) \geq 0$, for all $x\in \integ^8$. Moreover, $\hat{q}(x) = 0$ if and only if $x \in \integ \cdot (1,3,4,3,1,2,1,1)$. Since the graph of $\widehat{B}_8$ is tree, we conclude by \cite[Theorem 3.3]{tits}.
\end{proof}

\begin{remark}
In fact, using \cite[Theorem 2.3]{tits} we see from the proof above that the vectors in $\integ_{>0} \cdot (2,3,4,3,2,2)$ are precisely those dimension vectors for which there exists an infinite number of (pairwise non-isomorphic) indecomposable representations of $\EEE$.
\end{remark}

\section{The case of finitely many orbits}
\label{sec:finite}

In this section, assume that $X$ has \textit{finitely many orbits} under the action of a connected linear algebraic group $G$. In Subsection \ref{subsec-linred}, we assume in addition that $G$ is \emph{reductive}. 

\subsection{The quiver of the category}

For an orbit $O\subset X$, we denote by $T^*_O X$ the conormal bundle. Let $\Lambda=\Lambda(X,G)$ be 
\begin{equation}\label{eq:moment}
\Lambda = \bigcup_{O \subseteq X} \ove{T^*_O X} \subseteq T^* X.
\end{equation}
The moment map
\begin{eqnarray}\label{eqn-momentmap}
\mu\colon T^*X \to \lie^*
\end{eqnarray}
is induced by the map $\lie\to \calO_{T^*X}$ that sends a vector field to its symbol under the order filtration. Then $\Lambda = \mu^{-1}(0)$, where  $\mu^{-1}(0)$ denotes the set-theoretic fiber at $0$ of the moment map. Denote by $\op{mod}^{rh}_{\Lambda}(\dif_X)$ the full subcategory of $\op{mod}(\dif_X)$ of regular holonomic $\dif$-modules whose characteristic varieties are subsets of $\Lambda$. This is a Serre subcategory of $\op{mod}(\dif_X)$, also preserved under holonomic duality $\dual$.

\begin{theorem}\label{thm:eqfin}
\cite[Theorem 11.6.1]{htt} Let $G$ act on a (not necessarily smooth) variety $X$ with finitely many orbits. Then $\op{mod}_G(\dif_X)$ is a full subcategory of $\op{mod}^{rh}_{\Lambda}(\dif_X)$, preserved under holonomic duality $\dual$. The simple equivariant $\dif_X$-modules are labeled by pairs $(O , L)$, where $O\cong G/H$ is an orbit of $G$ in $X$ and $L$ is a finite-dimensional irreducible representation of the finite component group $H / H^0$ (here $H^0$ is the identity component subgroup of $H$).\qed
\end{theorem}

By the Riemann-Hilbert correspondence (see \cite[Chapter 7]{htt}), the category
$\op{mod}_G(\dif_X)$ is equivalent to the category of equivariant
perverse sheaves on $X$. The category $\op{mod}_G(\dif_X)$ is Artinian
and Noetherian with finitely many simples; moreover, it has enough
projectives, by \cite[Theorem 4.3]{vilon}. We give a constructive and
more elementary proof of this fact.

\begin{theorem}
\label{thm:findim}
With the assumptions as in Theorem \ref{thm:eqfin}, the category $\op{mod}_G(\dif_X)$ is equivalent to the category of finite-dimensional modules of a finite-dimensional algebra.
\end{theorem}

\begin{proof}
We proceed by induction on the number of orbits. If $X$ itself is an orbit $O\cong G/H$, then the result is trivial, since the correspondence in Theorem \ref{thm:eqfin} gives an equivalence of categories between $\op{mod}_G(\dif_X)$ and the category of finite-dimensional representations of the finite group $H/H^0$, which is a semi-simple category.

Now we turn to the general case. To prove the theorem, it is enough to show that $\op{mod}_G(\dif_X)$ has enough projectives (or injectives, by duality $\dual$), since then one can obtain the algebra by taking the (opposite) algebra of endomorphisms of the direct sum of all projective covers (see \cite[Lemma 1.2]{pervbgg}). Let $U$ be an open orbit in $X$, and write $Z=X\setminus U$. We know that $\op{mod}_G (\dif_U)$ is semi-simple, and by induction $\op{mod}_G (\dif_Z)$ has enough projectives. 

As usual, we assume we have an equivariant embedding $\phi\colon X\hookrightarrow Y$ with $Y$ smooth. Write $U' = Y\setminus Z$ and let $j:U' \to Y$ denote the open embedding. Recall the functor $j_! := \dual \circ j_* \circ \dual$ and the middle extension functor $j_{!*}$ which is the image of the natural map $j_!\to j_*$ (for more details, see \cite[Section 3.4]{htt}). Let $L'_1,L'_2,\dots, L'_k$ denote the simples in $\op{mod}_G (\dif_U)=\op{mod}_G^U (\dif_{U'})$. Then $L_i:=j_{!*} L'_i$ are precisely the simples in $\op{mod}_G^X (\dif_Y)$  with support equal to $X$ (see \cite[Theorem 3.4.2]{htt}). Using adjointness, we can see that $j_! L'_i$ is the projective cover of $L_i$, and $j_* L'_i$ is the injective envelope of $L_i$ in $\op{mod}_G ^X (\dif_Y)$ (see \cite[Lemma 2.4]{binary}). For each $i$ there is an exact sequence
\begin{equation}\label{eq:ex1}
0\to A_i \to j_! L'_i \to L_i \to 0,
\end{equation}
where $A_i$ is supported on $Z$. The following result shows that there are no extensions between the simples $L_1,\dots, L_k$:

\begin{lemma}\label{lem-noext}
  If $M_1', M_2'$ are simple objects in $\op{mod}_G(\dif_U)$, $U$ the open orbit, then the middle extensions $M_1=j_{!*}M_1'$ and $M_2=j_{!*}M_2'$ have $\Ext^1(M_1,M_2)=0$ in $\op{mod}_G(\dif_X)$.
\end{lemma}
\begin{proof}[Proof of Lemma \ref{lem-noext}]
 By the embedding of categories in Lemma \ref{lem:restrict} (with $H=G$), it is enough to show that $\Ext^1(M_1',M_2')=0$  in $\op{mod}_G(\dif_U)$. This follows as the category $\op{mod}_G(\dif_U)$ is semi-simple; see the start of the proof of the theorem.
  \end{proof}
We return to the proof of the theorem. For $M,N\in \op{mod}_G^X (\dif_Y)$ one can define the (Yoneda) extension groups $\Ext^i(M,N)$ in the Abelian category $\op{mod}_G^X(\dif_Y)$ for all $i\geq 0$. Since $\op{mod}_G^X (\dif_Y)$ is a fully faithful subcategory of $\op{mod}(\dif_Y)$, and since $M,N$ are holonomic by Theorem \ref{thm:eqfin}, both $\Ext^0(M,N)=\Hom_{\dif_Y}(M,N)$ and $\Ext^1(M,N)$ are finite-dimensional vector spaces  (see \cite[Theorem 4.45]{kashi}).

Now let $S_1,S_2,\dots,S_m$ denote the simples in $\op{mod}_G(\dif_Z)=\op{mod}_G^Z (\dif_X)$, and let $I_1,I_2,\dots,I_m$ denote their respective injective envelopes in $\op{mod}_G^Z (\dif_Y)$. By duality, we are left to show that each $S_i$ has an injective envelope in $\op{mod}_G^X (\dif_Y)$.

Put $I:=\bigoplus_{i=1}^m I_i$. Since $I$ is supported in $Z$, $j^*(I)=0$. Thus, $0=\Hom(j^*(\dual I),\dual L_i')=\Hom(\dual I,j_*\dual L_i)=\Hom(\dual I,\dual j_!L_i')=\Hom(j_!L_i',I)$. It follows, that \eqref{eq:ex1} induces $\Hom(A_i,I)=\Ext^1(L_i,I)$.

For each $i=1,2,\dots, k$, we choose basis elements $e^i_1,e^i_2,\dots, e^i_{a_i}$ of $\Ext^1(L_i,I)$, where \\$\dim \Ext^1(L_i,I)=a_i$, and represent each $e^i_j$ by an exact sequence
\[0\to I \to E_j^i \to L_i \to 0,\]
for some $E_j^i\in \op{mod}_G^X (\dif_Y)$. Taking the direct sum of all these sequences, we take its pushout via the diagonal projection onto $I$, obtaining the following diagram:
\[\xymatrix@R-1pc{
0 \ar[r] & \displaystyle\bigoplus_{i,j} I \ar[r] \ar[d]& \displaystyle\bigoplus_{i,j} E_j^i \ar[r] \ar[d]& \displaystyle\bigoplus_{i,j} L_i \ar[r]\ar@{=}[d] & 0 \\
0 \ar[r] & I \ar[r] & Q \ar[r]  & \displaystyle\bigoplus_{i,j} L_i \ar[r]  & 0.
}\]
For an arbitrary $i$ with $1\leq i \leq k$, we apply the functor $\Hom(L_i, - )$ to the bottom row, which yields the exact sequence
\[0\to \Hom(L_i, Q) \to \Hom(L_i,  \displaystyle\bigoplus_{j=1}^{a_i} L_i)\xrightarrow{d} \Ext^1(L_i, I) \to \Ext^1(L_i,Q) \to 0,\]
since there are no extensions between $L_1,\dots,L_k$. However, the map $d$ is an isomorphism by construction, hence $\Hom(L_j,Q)=\Ext^1(L_j,Q)=0$ for all $j=1,\dots, k$.

Also, applying the functor $\Gamma_Z$ to the bottom row, we obtain that $\Gamma_Z Q = I$.

Now, we construct inductively a sequence of $\dif$-modules $Q_0,Q_1, Q_2 \dots $ in $\op{mod}_G^X (\dif_Y)$ in the following way. Put $Q_0:=Q$. Assume we defined $Q_{p-1}$. If there is any simple $S:=S_j$ for some $1\leq j \leq m$, such that $\Ext^1(S,Q_{p-1})\neq 0$, then we define $Q_p$ by choosing a non-split exact sequence
\begin{equation}\label{eq:ex2}
0\to Q_{p-1} \to Q_p \to S \to 0.
\end{equation}
First, we claim that for any $p$ we have $\Gamma_Z Q_p=I$ and  $\Hom(L_i,Q_p)=\Ext^1(L_i,Q_p)=0$ for all $i=1,\dots k$. We have already shown  the claim for $p=0$. Now assume the claim is true for $Q_{p-1}$. Applying $\Gamma_Z$ to \eqref{eq:ex2} we get an exact sequence
$$0\to I \to \Gamma_Z Q_p \xrightarrow{\phi} S.$$
Assume the map $\phi$ is non-zero. Then $\phi$ must be surjective, since $S$ is simple. Hence we get an exact sequence in $\op{mod}_G^Z (\dif_Y)$, and since $I$ is injective in this category, the sequence must split. Choose a splitting $\psi: S \to \Gamma_Z Q_p$, so $\phi \circ \psi = id$. Then composing $\psi$ with the inclusion $\Gamma_Z Q_p \hookrightarrow Q_p$ yields a splitting of \eqref{eq:ex2}, which is a contradiction. Hence $\phi$ must be zero, \emph{i.e.} $\Gamma_Z Q_p = I$.

Applying $\Hom(L_i,-)$ to \eqref{eq:ex2} we get $\Hom(L_i,Q_p)=0$, for any $i=1,\dots,k$. Now we apply $\Hom(-,Q_p)$ to the sequence \eqref{eq:ex1} and obtain an exact sequence
$$0\to \Hom(j_! L_i', Q_p) \to \Hom(A_i, Q_p) \to \Ext^1(L_i,Q_p) \to 0,$$
since $j_! L_i'$ is projective. By construction, $\dim \Hom(j_! L_i',
Q_p) = \dim\Hom(j_!L_i',Q)=a_i$ and $\dim \Hom(A_i,Q_p)=\dim \Hom(A_i,\Gamma_Z Q_p) = \dim \Hom(A_i, I)=a_i$. This implies $\Ext^1(L_i,Q_p)=0$, giving the last part of the claim.

Now let $P:= \dual(I)$. We show that the process of constructing $Q_0,Q_1,Q_2,\dots$ stops after $\dim \Ext^1(P,Q)$ steps, \emph{i.e.} $Q_{\dim \Ext^1(P,Q)}$ is injective in $\op{mod}_G^X (\dif_Y)$. More precisely, applying $\Hom(P,-)$ to the sequence \eqref{eq:ex2}, we get the exact sequence
$$0\to\Hom(P,Q_{p-1})\to \Hom(P,Q_p)\to \Hom(P,S) \to \Ext^1(P,Q_{p-1}) \to \Ext^1(P,Q_{p})\to 0,$$
since $P$ is projective in $\op{mod}_G^Z (\dif_Y)$. We have
$\,\dim\Hom(P,S) = 1$, and $\Hom(P,Q_{p-1})=\Hom(P,\Gamma_Z Q_{p-1}) = \Hom(P,I)=\Hom(P,Q_p)$, where the last equality follows since $I=\Gamma_Z Q_p$ and $P$ has support in $Z$. Hence $\dim \Ext^1(P,Q_{p}) = \dim \Ext^1(P,Q_{p-1})-1$, and the process stops precisely when $Q_p$ is injective, so when $p=\dim \Ext^1(P,Q)$ and $\Ext^1(P,Q_{p})=0$. 
\end{proof}

\begin{remark}\label{rem:fundfin}
If $X$ is $\dif$-affine and $G$ reductive, Theorem \ref{thm:findim} follows also from Proposition \ref{prop:proj}.

In the non-equivariant setting, suppose there is a Whitney stratification on an algebraic variety $X$ such that each stratum has a finite fundamental group, and let $\Lambda$ denote the union of the closures of the conormal bundles of the strata. Then $\op{mod}^{rh}_{\Lambda}(\dif_X)$ is equivalent to the category of finite-dimensional modules of a finite-dimensional algebra. This generalizes \cite[Theorem 1.1]{pervbgg}.
\end{remark}

\vspace{0.1in}

By the considerations in Section \ref{subsec:quiv}, the categories considered in Theorem \ref{thm:findim} are in turn equivalent to the category $\rep(\widehat{Q})$ of finite-dimensional representations of a quiver $\widehat{Q}$ . One of the general goals of the paper is the following:

\begin{prb}Determine the quiver with relations $\widehat{Q}$ such that $\op{mod}_G(\dif_X)\cong \rep(\widehat{Q})$. 
\end{prb}

Such a quiver $\widehat{Q}$ is called the \emph{quiver of $\op{mod}_G(\dif_X)$}. For the construction of the quiver $\widehat{Q}$ of a finite-dimensional algebra (hence of the category $\op{mod}_G(\dif_X)$), we refer the reader to \cite{elements}. In particular, the vertices of $\widehat{Q}$ correspond to the simple equivariant $\dif_X$-modules, and the number of arrows from vertex $x$ to vertex $y$ is equal to $\dim_\complex \Ext^1 (S_x, S_y)$, the dimension of the extension group (in $\op{mod}_G(\dif_X)$) between the corresponding simple $\dif_X$-modules $S_x,S_y$. Since simples have no self-extensions in $\op{mod}_G(\dif_X)$ (see Lemma \ref{lem-noext}) the quiver has no loops.

\begin{corollary}\label{cor:relations}
Let $\widehat{Q}$ be the quiver of $\op{mod}_G(\dif_X)$ as in Theorem \ref{thm:findim}. Suppose $\ove{O}$ is an irreducible component of $X$, for some $G$-orbit $O$, and assume that $S$ is a simple in $\op{mod}_G(\dif_X)$ with support $\ove{O}$. If there exits (up to relations) a non-trivial path from $S$ to another simple $S'$ in $\widehat{Q}$, then the support of $S'$ is contained in $\ove{O}\setminus O$.
\end{corollary}

\begin{proof}
By duality $\dual$, it is enough to show that if the support of $S'$ is not contained in $\ove{O}\setminus O$, then there are no non-trivial paths from $S'$ to $S$. This follows by the construction of the injective envelope $j_*j^*S$ from the proof of Theorem \ref{thm:findim} (we use the same notation, replacing $Z$ in that proof with $\ove{O}\setminus O$ here), since all the simple composition factors of $j_*j^*S/S$ are supported on  $\ove{O}\setminus O$. Hence, if $S'\cong S$, there is only the trivial path (of zero length) in $\widehat{Q}$ between $S$ and $S$, and if $S' \not\cong S$ then there are no non-zero paths between $S'$ and $S$ up to relations.
\end{proof}

Recall the moment (conormal) variety $\Lambda(X,G)$ from \eqref{eq:moment}.
\begin{lemma}\label{lem:simplyconn}
Additionally to the assumptions in Theorem \ref{thm:eqfin}, suppose that $G$ is semi-simple and simply connected. Then $\op{mod}_G(\dif_X) = \op{mod}_\Lambda^{rh} (\dif_X)$.
\end{lemma}

\begin{proof}
  Since $G$ is semi-simple, $\op{mod}_G(\dif_X)$ is a subcategory of  $\op{mod}_\Lambda^{rh} (\dif_X)$ that is closed under extensions  by Proposition \ref{prop:serre}. Hence, it is enough to show that any simple $\dif$-module in $\op{mod}_\Lambda^{rh} (\dif_X)$ is $G$-equivariant. 

Recall that regular simple $\dif_X$-modules come from Deligne systems: irreducible local systems $L$ on locally closed submanifolds $C$; these are in 1-to-1 correspondence with the simple representations of the fundamental group $\pi_1(C)$. The characteristic variety of the $\dif_X$-module corresponding to $L$ contains as one component the conormal to $C$. Since this is supposed to be in $\Lambda$, only $G$-orbits qualify for $C$. It remains to show that $L$ is $G$-equivariant. Since $G$ is connected and simply connected, the fundamental group $\pi_1(C)$ of an orbit $C=G/H$ is isomorphic to the component group $H/H^0$. The claim follows by Theorem \ref{thm:eqfin}. 
\end{proof}

\medskip

We return to $X$ being smooth and formulate the following result that we use in Section \ref{sec:examples}. 

\begin{lemma}\label{lem:local}
Suppose $X$ is affine and $O\subset X$ a $G$-orbit of codimension $c$. Let $Z=\ove{O}\setminus O$ and denote $U=X\setminus Z$. Then the $\dif_X$-module $\displaystyle{H^{c}_O(U,\orb_U)}$ is the injective envelope in $\op{mod}_G^{\ove{O}}(\dif_X)$ of the simple $\dif_X$-module corresponding to the trivial connection on $O$. Moreover, we have an exact sequence of $\dif_X$-modules
\[ 0 \to H^c_{\ove{O}} (X,\orb_X) \to H_O^c(U,\orb_U) \to H_Z^{c+1}(X,\orb_X) \to  H^{c+1}_{\ove{O}} (X,\orb_X).\]
\end{lemma}

\begin{proof}
Since $O$ is smooth and closed in $U$, $\mathcal{H}_O^c (\orb_U)$ is the simple $\dif_U$-module corresponding to the trivial connection on $O$ (see \cite[Proposition 1.7.1]{htt}), and $\mathcal{H}_O^i (\orb_U)=0$ for  $i\neq c$. In particular, the spectral sequence $H^i (\mathcal{H}_O^j (\orb_U)) \Rightarrow H^{i+j}_O (U,\orb_U)$ degenerates. By \cite[Lemma 2.4]{binary}, this shows that $\displaystyle{H_O^c(U,\orb_U)}$ is the injective envelope in $\op{mod}_G^{\ove{O}}(\dif_X)$.
\end{proof}

We need the following result on the geometry of the zero-fiber of the moment map  $\mu : T^* X \to \lie^*$ in \eqref{eqn-momentmap}. Recall that in our situation of finitely many orbits, the irreducible components of the zero-fiber \eqref{eqn-momentmap} are just closures of the conormal bundles of the orbits. We have the following: 

\begin{lemma}\label{lem:moment}
Let $Z$ be an irreducible component of the scheme-theoretic fiber $\mu^{-1} (0)$, and $x\in Z$ a closed point. Then $x$ is a smooth point of $\mu^{-1}(0)$ if and only if $\ove{G\cdot x}=Z$.
\end{lemma}

\begin{proof} We denote temporarily $T^*X$ by $M$. Differentiating the action of $G$ on $M$ induces a map $\lie \to \op{Vect}(M)$, denoted $\xi \mapsto \xi_M$. 

The moment map $\mu$ is induced from the map $\lie \to \Gamma(M,\orb_{M})$, denoted $\xi \mapsto H_{\xi}$. (Vector fields are functions on $T^*X$, hence elements of $\calO_M$). Take a point $x\in Z$. Then $\mu^{-1}(0)$ is smooth at $x$ if and only if $\dim \op{span} \{dH_{\xi,x} \}_{\xi \in \lie} = \dim M - \dim Z = \dim Z$. On the other hand, the canonical non-degenerate symplectic form on $M$ induces an isomorphism $T_x^* M \to T_x M$, sending $dH_{\xi,x}\mapsto\xi_{M,x}$ (since the action of $G$ on $M$ is Hamiltonian). Hence, $\dim \op{span} \{dH_{\xi,x} \}_{\xi \in \lie}= \dim \op{span} \{\xi_{M,x}\}_{\xi \in \lie} = \dim T_{x} (G\cdot x)$, finishing the proof.
\end{proof}

\subsection{The linear reductive case}
\label{subsec-linred}

We assume for the rest of the section that
$G$ is a \emph{connected linear reductive group} acting on an irreducible smooth variety $X$ with finitely many orbits.

Recall that the characteristic cycle $\charC(M)$ of a $\dif$-module $M$ is the formal linear combination of the irreducible components of its characteristic variety counted with multiplicities (see \cite[Section 2.2]{kashi}).
\begin{notation}
  For an irreducible subvariety $Z$ of $T^* X$, and a $\dif_X$-module (or a $\orb_{T^*X}$-module) $M$, let \textit{multiplicity} $\op{mult}_Z M\in \nat$ of $M$ at $Z$ be the multiplicity of $Z$ in the characteristic cycle $\charC(M)$.

If $M$ is a holonomic $\dif_X$-module, and $S$ is a simple $\dif_X$-module, then we denote by $[M:S]$ the number of factors isomorphic to $S$ in a composition series of $M$. 
\end{notation}
  
Note that Lemma \ref{lem:moment} implies that $Z$ has a dense $G$-orbit if and only if $\op{mult}_Z \orb_{ \mu^{-1} (0)}=1$. 

Now let $S_1,\dots S_n$ be all the (pair-wise non-isomorphic) simple equivariant $\dif_X$-modules. Suppose that the support of $S_i$ is the orbit closure $\ove{O_i}$, and denote $Z_i = \ove{T^*_{O_i} X}$.

\begin{proposition}\label{prop:multfinite}
Let $M\in \op{mod}_G(\dif_X)$. Then $\Gamma(X,M)$ is a multiplicity-finite $G$-module. Moreover, for any irreducible $G$-module $V_\lambda$ corresponding to a dominant weight $\lambda\in \Pi$, we have the bound
\[m_{\lambda}(\Gamma(X,M))\leq  \left(\mathlarger{\displaystyle\sum_{i=1}^n} [M:S_i] \cdot \left\lfloor\dfrac{\dim V_\lambda\cdot\op{mult}_{Z_i} \orb_{\mu^{-1}(0)}}{\op{mult}_{Z_i} S_i}\right\rfloor\right).\]
\end{proposition}

\begin{proof}
Take an arbitrary $\lambda \in \Pi$. By Lemma \ref{lem:rep} (c) and Theorem \ref{thm:eqfin} both $P(\lambda)$ and $M$ are regular holonomic, hence $\Hom_{\dif_X}(P(\lambda), M)$ is finite-dimensional (see \cite[Theorem 4.45]{kashi}).  This implies by Lemma \ref{lem:rep} (a) that $m_\lambda(\Gamma(X,M))$ is finite-dimensional. Clearly, it is enough to prove the result on the bound when $M=S_i$ is simple, for some $i$. For simplicity, put $S=S_i$ and $Z=Z_i$. By Schur's lemma we have $\End_{\dif_X}(S)=\complex$, and so
\begin{equation}\label{eq:multhom}
m_\lambda(\Gamma(X,S)) = \dim \Hom_{\dif_X} (P(\lambda),S)\leq [P(\lambda):S] \leq \left\lfloor\dfrac{\op{mult}_Z P(\lambda)}{\op{mult}_Z S}\right\rfloor.
\end{equation}
We are left to show that $\op{mult}_Z P(\lambda) \leq \dim V_\lambda\cdot\op{mult}_{Z} \orb_{\mu^{-1}(0)}$. Take the following filtration on $P(\lambda)$:
$$F_k (P(\lambda)) := F_k \dif_X \cdot (1 \otimes V_\lambda) \subset \dif_X \otimes_{U(\lie)} V_\lambda, \text{ for } k\in \nat,$$
where the filtration $F_k \dif_X$ is given by the usual order-filtration on $\dif_X$. This is a coherent (or good) filtration in the sense of \cite[Section 2.2]{kashi}, and let $\op{gr}(P(\lambda))$ denote of the associated graded sheaf, viewed as an $\orb_{T^* X}$-module via pulling back from $X$ to $T^* X$.

The symbol of each Lie algebra element under the order filtration
acts on $\op{gr}(P(\lambda))$ as zero (see (\ref{eq:action2})). Since the scheme $\mu^{-1}(0)$ is defined by these symbols, the associated graded of the natural map $\dif_X \otimes_\complex V_\lambda \twoheadrightarrow P(\lambda)$ factors through a surjective map of $\orb_{T^* X}$-modules 
$$ \orb_{\mu^{-1}(0)}\otimes_{\complex} V_\lambda \twoheadrightarrow \op{gr}(P(\lambda)).$$
Since $Z$ is an irreducible component of $\mu^{-1}(0)$, this shows  $\op{mult}_Z P(\lambda) \leq \dim V_\lambda\cdot\op{mult}_{Z} \orb_{\mu^{-1}(0)}$. 
\end{proof}

\medskip

\begin{definition}
For a (smooth, irreducible) $G$-variety $X$, we say that $X$ is
\textit{spherical}, if $X$ has an open $B$-orbit, for a Borel subgroup
$B$ of $G$. This is equivalent to $X$ having finitely many $B$-orbits (see \cite{brion} or \cite[Theorem 1]{Vinberg1986}).
\end{definition}
If $X$ is quasi-affine and $\Gamma(X,\orb_X)$ is multiplicity-bounded, then $X$ must be spherical by \cite[Proposition 2.4.]{alebri}; we provide the following result in the reverse direction:

\begin{theorem}\label{thm:multfreee}
Let $X$ be a spherical $G$-variety and $S$ be a simple equivariant
$\dif_X$-module. Then: 
\begin{itemize}
\item[(a)] $\Gamma(X,S)$ is a multiplicity-free $G$-module.
\item[(b)] The characteristic cycle of $P(\lambda)$ is multiplicity-free, for any $\lambda\in \Pi$.
\item[(c)] If $\Gamma(X,S)\neq 0$, then the characteristic cycle of $S$ is multiplicity-free.
\end{itemize}
\end{theorem}

\begin{proof}
The support of $S$ is the closure $\ove{O}$ of some $G$-orbit
$O$. We put $Z=\ove{T^*_O X}$; so $\op{mult}_Z S\geq 1$.  Fix any
$\lambda\in\Pi$. By Lemma \ref{lem:rep} (a), we have $m_\lambda(\Gamma(X,S))= \dim
\Hom_{\dif_X} (P(\lambda),S)$. As in (\ref{eq:multhom}) this number is bounded
above by  $\leq \op{mult}_Z P(\lambda)$. Hence, part (b) implies part (a), and we show now that $\op{mult}_Z P(\lambda)\leq 1$.

 Recall that the highest weight vector $v_\lambda$ generates $P(\lambda)$ globally as in \eqref{eq:highest}. We consider the following filtration on $P(\lambda)$:
\[
F_k (P(\lambda)) := F_k \dif_X \cdot (1 \otimes v_\lambda), \text{ for } k\in
\nat,
\]
where $F_k \dif_X$ denotes the usual order-filtration on $\dif_X$. Fix
now a concrete $B$ that witnesses the sphericality of $X$. Let $\op{gr}(P(\lambda))$ denote the associated graded sheaf (pulled back from $X$ to $T^* X$) and let $\orb_{\mu^{-1}_{\mathfrak{b}}(0)}$ be the structure sheaf of the scheme-theoretic fiber of the moment map $\mu_{\mathfrak{b}}: T^*X \to \mathfrak{b}$ at $0$, where $\mathfrak{b}=\op{Lie}(B)$. 

Since $v_\lambda$ is a highest weight vector, the symbol of each Lie algebra element of $\mathfrak{b}$ under the order filtration
acts on $\op{gr}(P(\lambda))$ as zero (see (\ref{eq:action2})). Since the scheme $\mu_{\mathfrak{b}}^{-1}(0)$ is defined by these symbols, the associated graded of the map $\dif_X \twoheadrightarrow P(\lambda)$ (sending $1 \mapsto 1\otimes v_\lambda$) factors through a surjective map of $\orb_{T^* X}$-modules 
$$ \orb_{\mu_{\mathfrak{b}}^{-1}(0)}\twoheadrightarrow \op{gr}(P(\lambda)).$$
Since $B$ acts on $X$ with finitely many orbits, $Z$ is an irreducible component of $\mu_{\mathfrak{b}}^{-1}(0)$ and $\op{mult}_Z P(\lambda) \leq  \op{mult}_Z \orb_{\mu_{\mathfrak{b}}^{-1}(0)}$. Moreover, $B$ acts on $Z$ with a dense orbit (see \cite[Corollary 2.4]{pany}), so Lemma \ref{lem:moment} yields $\op{mult}_Z \orb_{\mu_{\mathfrak{b}}^{-1}(0)}=1$, finishing the proof of (b).

For part (c), since $\Gamma(X,S)\neq 0$, there is a $\lambda\in \Pi$ such that $m_\lambda(\Gamma(X,S))=1$. By Lemma (\ref{lem:rep}) (a), this gives a surjection $P(\lambda)\twoheadrightarrow S$, hence the claim follows from (b).
\end{proof}

Theorem \ref{thm:multfreee} (a) implies that  if $M$ is an equivariant $\dif$-module on a spherical $X$, then $\Gamma(X,M)$ is a multiplicity-bounded $G$-module (say, with the uniform bound equal to the length of $M$).

\begin{remark}
When $X$ is additionally $\dif$-affine, one can give a somewhat simpler proof of part (a) of the above theorem using the Jacobson Density Theorem on the simple $\dif$-module $S$ (since $\End_{\dif}(S)=\CC$), and that $\Gamma(X,\dif_X)^G$ is a commutative algebra (this follows as in \cite[Proposition 7.1]{howumed}). For a proof in the same spirit, see Theorem \ref{thm:capelli} or \cite[Proposition 7.1]{howumed}.
\end{remark}

When $X$ is a $\dif$-affine spherical variety, we obtain several additional results.

\begin{corollary}\label{cor:projmult}
Assume $X$ is a $\dif$-affine spherical variety, and let $P$ be an indecomposable projective (resp.\ injective) in $\op{mod}_G(\dif_X)$. Then the characteristic cycle of $P$ is multiplicity-free.
\end{corollary}

\begin{proof}
By duality $\dual$, it is enough to consider the case when $P$ projective, in which case it is the projective cover of a simple equivariant $\dif$-module $S$, say. Since $X$ is $\dif$-affine, there is a $\lambda\in \Pi$ with $m_\lambda(S)=1$. Then by Lemma \ref{lem:rep} and Proposition \ref{prop:proj}, we have that $P$ is a direct summand of $P(\lambda)$. This concludes the proof by Theorem \ref{thm:multfreee} (b).
\end{proof}

We record a consequence of Corollary \ref{cor:projmult} and the construction of $\widehat{Q}$. Section \ref{sec:examples} is a case-by-case study of this class of quivers.
\begin{corollary}\label{cor:quivnice}
Assume $X$ is a $\dif$-affine spherical variety, and let $\widehat{Q}$ be the quiver such that $\op{mod}_G(\dif_X)\cong \rep(\widehat{Q})$. Then the number of paths (up to relations)  from vertex $i$ to vertex $j$  in $\widehat{Q}$ is at most 1.
\end{corollary}
\begin{proof}
  The number of such paths (up to relations) agrees with the number of copies of the simple at vertex $j$ in a composition chain for the projective cover of the simple at vertex $i$.
  \end{proof}

Denote by $\mathcal{Z}U(\mathfrak{g})$ the center of the universal enveloping algebra of $\mathfrak{g}$, and let $\dif_X^G$ denote the algebra of invariant differential operators. Recall the map $U(\mathfrak{g}) \to \dif_X$. Clearly, this induces a map
\[\rho: \mathcal{Z}U(\mathfrak{g})  \to \dif_X^G.\]
Following \cite{howumed}, we introduce the following.
\begin{definition}\label{def:capelli}
For a $\dif$-affine spherical $G$-variety $X$, we say that $X$ of \textit{Capelli type} if the map $\rho$ considered above is surjective.
\end{definition}

The irreducible representations $X$ of a reductive group $G$ that are of Capelli type were classified in \cite{howumed}. 

When $X$ is of Capelli type, we have the following is stronger version of Theorem \ref{thm:multfreee} (a).

\begin{theorem}\label{thm:capelli}
Assume $X$ is of Capelli type, and let $S$ be a simple equivariant $\dif_X$-module. Take any $\lambda \in \Pi$ such that $m_{\lambda}(\Gamma(X,S)) \neq 0$. Then $P(\lambda)$ is the projective cover of $S$ in $\op{mod}_G(\dif_X)$, and $\Gamma(X,P(\lambda))$ is a multiplicity-free $G$-module.
\end{theorem}

\begin{proof}
By Lemma \ref{lem:rep} and Proposition \ref{prop:proj}, the projective cover of $S$ is a direct summand of $P(\lambda)$. To show that $P(\lambda)$ is indecomposable, by Lemma \ref{lem:rep} (a) it is enough to show that $m_\lambda(P(\lambda))= 1$.

Pick a basis $v_1,\dots, v_k$ of the vector space $V_\lambda$ such that, say, $v_1 = v_\lambda$ is the highest weight vector. Since $V_\lambda$ is a simple $U(\lie)$-module, for any $i=1,\dots,k$ we have by the Jacobson Density Theorem an element $\xi_i \in U(\lie)$ such that $\xi_i \cdot v_i = v_1$ and $\xi_i \cdot v_j = 0$, for $j\neq i$ (since the elements $v_1,\dots, v_k$ are linearly independent over $\End_{U(\lie)}(V_\lambda)=\CC$).

Let $W$ be a $G$-stable subspace of $P(\lambda)$ that is $G$-isomorphic to the irreducible $V_\lambda$, and let $\phi : V_\lambda \to W$ denote this isomorphism. Put $w_i := \phi (v_i)$, for $i=1,\dots, k$. As in \ref{eq:highest}, the element $1\otimes v_\lambda$ generates the $\dif$-module $P(\lambda)$. Therefore, for any $i =1,\dots, k$, there is an element $d_i \in \dif_X$ such that $d_i \cdot (1 \otimes v_\lambda) = w_i$. Take the element $d'= \ds_{i=1}^k d_i \xi_i \in \dif_X$. By construction, we have $d' \cdot (1\otimes v_i) = w_i$, for all $i=1,\dots, k$, hence $d' |_{V_\lambda} = \phi$. Because $G$ is reductive, by averaging we can produce an element $d\in \dif_X^G$, with $d |_{V_\lambda} = \phi$.  Since $X$ is of Capelli type, there is an element $u \in \mathcal{Z}U(\mathfrak{g})$ such that $\rho(u)=d$. But then there is a constant $c\in \CC$ such that for any $i=1,\dots,k$ we have, since $u\in\mathcal{Z}(U(\lie))$,
\[w_i = d \otimes v_i = \rho(u) \otimes v_i = 1 \otimes (u \cdot v_i) = c \cdot (1\otimes v_i),\]
which gives $W=V_\lambda, \phi = c \cdot \op{id}$, and so $m_\lambda(P(\lambda)) = 1$.

Now, since $P(\lambda)$ is the projective cover of $S$, by Lemma \ref{lem:rep} we have that for any other simple $G$-equivariant $\dif$-module $S'$ not isomorphic to $S$, we must have $m_\lambda(S')=0$. This shows that the space of sections of non-isomorphic simple equivariant $\dif$-modules admit no common irreducible $G$-modules. By Theorem \ref{thm:multfreee} (a)--(b), this implies that $P(\lambda)$ is a multiplicity-free $G$-module.
\end{proof}

In particular, we see from the above that if $m_{\lambda'}(\Gamma(X,S))\neq 0$ for another $\lambda'\in \Pi$, then $P(\lambda)\cong P(\lambda')$ as $\dif_X$-modules. 

\begin{corollary}\label{cor:nocommon}
Let $X$ be a spherical variety, and $S$ and $S'$ be non-isomorphic simple equivariant $\dif_X$-modules with projective covers $P,P'\in \op{mod}_G(\dif_X)$, respectively. Assume that one of the following holds:
\begin{itemize}
\item[(a)]  $\charC(S)$ and $\charC(S')$ have a common irreducible component,
\item[(b)] $X$ is $\dif$-affine and $\charC(P)$ and  $\charC(P')$ have a common irreducible component, or
\item[(c)] $X$ is of Capelli type.
\end{itemize}

Then there are no common irreducible $G$-modules in $\Gamma(X,S)$ and $\Gamma(X,S')$.
\end{corollary}

\begin{proof}
\begin{itemize}
\item[(a)] By contradiction, assume that there is a weight $\lambda \in \Pi$ such that $m_\lambda(\Gamma(X,S)) = m_\lambda(\Gamma(X,S')) = 1$. Let $Z$ be a common irreducible component of $\charC(S)$ and $\charC(S')$. By Lemma (\ref{lem:rep}), both $S$ and $S'$ are composition factors of $P(\lambda)$, giving $\op{mult}_{Z} P(\lambda)\geq 2$, contradicting Theorem \ref{thm:multfreee} (b).
\item[(b)] Let $Z$ be the common irreducible component of $\charC(P)$ and  $\charC(P')$ and assume by contradiction that there is a $\lambda\in\Pi$ such that $m_\lambda(S)=m_\lambda(S')=1$. Since $X$ is $\dif$-affine,  Lemma \ref{lem:rep} implies that $P(\lambda)$ surjects onto $S,S'$ respectively. By Proposition \ref{prop:proj}, $P(\lambda)$ is projective,  and so the projections from $P(\lambda)$ to $S$ and $S'$ must pass through the projective covers $P,P'$, which hence are direct summands of  $P(\lambda)$. This implies $\op{mult}_{Z} P(\lambda)\geq 2$, contradicting Theorem \ref{thm:multfreee} (b).
\item[(c)] Follows from the proof of Theorem \ref{thm:capelli}.
\end{itemize}
\end{proof}

In particular, if $X$ is of Capelli type this implies that for a simple equivariant $\dif_X$-module $S$ and $\lambda \in \Pi$ such that $m_\lambda(\Gamma(X,S))\neq 0$, we have $[M:S] = m_\lambda(\Gamma(X,M))$, for any equivariant $\dif_X$-module $M$. This shows that for spherical $\dif_X$-affine varieties of Capelli type, the $G$-decomposition of an equivariant $\dif_X$-module completely determines its composition series as a $\dif_X$-module.

For the explicit $G$-decompositions of the simple equivariant $\dif_X$-modules in the case of  $m\times n$ matrices, skew-symmetric matrices and symmetric matrices (all three of Capelli type) and applications to local cohomology, we refer the reader to the papers \cite{perlmana,lyubez,claudiu1,claudiu2,claudiu4,claudiu3}.

\section{Some techniques for equivariant $\dif$-modules}
\label{sec:tech}

Unless specified otherwise, we assume in this section that a connected linear algebraic group $G$ acts on an irreducible smooth variety $X$. In order to describe the categories of equivariant $\dif$-modules in some concrete cases, we develop some useful techniques.

\subsection{Reduction methods} \label{subsec:reduc}

For a closed $G$-stable subset $Y$ of $X$, we denote by $\op{mod}_G(\dif_X)_Y$ the full subcategory of $\op{mod}_G(\dif_X)$ consisting of equivariant $\dif_X$-modules $M$ that do not have as composition factors any simple $\dif_X$-modules $S$ with $\op{supp} S\subset Y$. Clearly, for any such $M$ we have $\Gamma_Y(M)=0$.

\begin{lemma}\label{lem:restrict}
Let $H$ be a connected closed algebraic subgroup of $G$, and Z an $H$-stable closed subset of $X$. Let $Y$ be the maximal $G$-stable closed subset of $Z$, put $U= X\setminus  Z$, and denote by $j:U \to X$ the open embedding.
Then $j^*$ induces an embedding of categories (\emph{i.e.} fully faithful exact functor)
$$j^*: \op{mod}_G(\dif_X)_Y\to \op{mod}_H(\dif_{U}).$$
\end{lemma}

\begin{proof}
The only thing we need to prove is that $j^*$ is fully faithful. Take $A,B\in \op{mod}_G(\dif_X)_Y$. By adjunction, we have
$$\Hom_{\dif_U}(j^*A,j^*B)\cong \Hom_{\dif_X}(A,j_* j^* B).$$
Since $\Gamma_Z(B) = \Gamma_Y(B)=0$, we have an exact sequence
\begin{gather}\label{B-sequence}
0\to B\to j_* j^*B \to \mathcal{H}^1_Z(B)\to 0.
\end{gather}
Since $\op{supp} \dual\mathcal{H}^1_Z(B)=\op{supp} \mathcal{H}^1_Z(B) \subset Z$, we have 
\[
\Hom_{\dif_X}(A,\mathcal{H}^1_Z(B))\cong\Hom_{\dif_X}(\dual\mathcal{H}^1_Z(B),\dual
A)= \Hom_{\dif_X}(\dual\mathcal{H}^1_Z(B),\Gamma_Z\dual A)=0.
\]
Applying the functor $\Hom_{\dif_X}(A,-)$ to the sequence
\eqref{B-sequence} we get
$$\Hom_{\dif_X}(A, B)\cong \Hom_{\dif_X}(A,j_* j^* B),$$
proving the claim.
\end{proof}

\begin{remark}\label{rem:subquiv}
Assume that $G$ (resp.\ $H$) acts on $X$ (resp.\ $U$) with finitely many orbits (so that the associated quiver is finite). Since the embedding $j^*$ above sends simples to simples, in this case Lemma \ref{lem:restrict} implies that the quiver corresponding to $\op{mod}_G(\dif_X)_Y$ is a subquiver (as defined in Section \ref{subsec:quiv})  of the quiver corresponding to $\op{mod}_H(\dif_{U})$.
\end{remark}

\begin{lemma}\label{lem:faithful}
Let $G$ be a (not necessarily connected) linear algebraic group acting on a $\dif$-affine variety $X$. Let $K$ denote the kernel of the action map (so $K$ acts on $X$ trivially), with identity component $K^0$. Then $\op{mod}_G (\dif_X)\cong \op{mod}_{G/K^0}(\dif_X)$.

Furthermore, denote by $\chi_1, \dots, \chi_r$ all the distinct isomorphism classes of irreducible representations of $K/K_0$ . Then $\op{mod}_G (\dif_X)$ is equivalent to the product of categories
\[\op{mod}_{G}(\dif_X) \cong \prod_{i=1}^r \op{mod}^{\chi_i}_{G}(\dif_X),\]
where $\op{mod}^{\chi_i}_{G}(\dif_X)$ denotes the full subcategory of $\op{mod}_G (\dif_X)$ consisting of $\dif$-modules $M$ for which $M=M_{(\chi_i)}$ (the $\chi_i$-isotopic component). 
\end{lemma}

\begin{proof}
The Lie algebra of $K^0$ acts trivially on any $\dif_X$-module. This shows that the natural functor $\op{mod}_{G/K^0}(\dif_X) \to \op{mod}_G (\dif_X)$ is an equivalence of categories.

For any $M\in \op{mod}_{G}(\dif_X)$ we can consider its $K/K_0$-isotypical decomposition $M\cong \oplus_{i=1}^r M_{(\chi_i)}$. Since $M$ is in particular a $K$-equivariant $\dif$-module and $K$ acts on $X$ trivially, each $M_{(\chi_i)}$ is a $\dif_X$-module. As a morphism in $\op{mod}_{G}(\dif_X)$ is $G$-equivariant, it respects $K$-isotypical components. This completes the proof.
\end{proof}

The following result is well-known. For the case of finitely many $G$-orbits, this follows directly according to Theorem \ref{thm:eqfin} from the corresponding result on perverse sheaves \cite[Proposition 5.1]{berlun}). 

\begin{proposition}\label{prop:induct}
Let $H$ be a closed algebraic subgroup of $G$, with $G,H$ not necessarily connected.
\begin{itemize}
\item[(a)] If $H\subset G$ is a normal subgroup acting on $X$ freely with $\pi: X\to Y$ a principal $H$-bundle, then we have an equivalence of categories
\[\op{mod}_G(\dif_{X})\cong \op{mod}_{G/H}(\dif_Y).\]
\item[(b)] If $Z$ is a smooth $H$-variety, then we have an equivalence of categories $$\op{mod}_H (\dif_ Z)\cong \op{mod}_{G}(\dif_{G\times_H Z}).$$
\end{itemize}
\end{proposition}

\begin{proof}
For part (a), we note that the map $\pi$ is $G$-equivariant. Since $H$ acts on $Y$ trivially, this implies that we have a well-defined functor $\pi^*:  \op{mod}_{G/H}(\dif_Y) \to \op{mod}_G(\dif_{X})$. Similarly, we have a well-defined functor $\pi_*^H:\op{mod}_G(\dif_{X})\to \op{mod}_{G/H}(\dif_Y)$ by taking $H$-invariant sections. We claim that the functors $\pi^*$ and $\pi_*^H$ are inverses to each other, for which it is enough to consider the case $G=H$. The claim is clear if $\pi$ is trivial, when  $X = H\times Y$. The general case can be reduced to the trivial one, as $\pi$ is locally trivial in the \'etale topology \cite[Section 4.8]{popvin}.

Part (b) follows from part (a) by the equivalences
\[\op{mod}_{G}(\dif_{G\times_H Z}) \cong \op{mod}_{G\times H} (\dif_{G\times Z}) \cong \op{mod}_H (\dif_Z).\]
\end{proof}

In nice cases, we have can make a reduction as follows.

\begin{proposition}\label{prop:redaff}
Assume $X$ is affine and $G$ reductive acting on $X$ with finitely many orbits. Then there is a (not necessarily connected) reductive subgroup $H$ of $G$ and an $H$-module $V$ such that 
$$\op{mod}_G (\dif_X) \cong \op{mod}_H (\dif_V).$$
\end{proposition}

\begin{proof}
On the open $G$-orbit, $G$-invariant functions are constant. So $\complex[X]^G=\complex$, and $X$ has a unique closed orbit. Let $H$ be the stabilizer of an element of the closed orbit. Then we have a $G$-equivariant isomorphism $X\cong G\times_H V$, for some $G$-module $V$ (see \cite[Theorem 6.7]{popvin}). The result follows by Proposition \ref{prop:induct}. 
\end{proof}

Hence, for the rest of the paper we focus on connected reductive groups acting linearly on vector spaces.

\subsection{Bernstein--Sato polynomials}\label{subsec:bs}

Let $X$ be a vector space with a linear action of a connected reductive group $G$. In this section, we assume that $X$ is a \textit{prehomogeneous} vector space, that is, $X$ has an open dense $G$-orbit (see \cite{saki}); this holds for example if $X$ has finitely many $G$-orbits. We discuss the notion of Bernstein--Sato polynomials in this setting. For more details, see \cite{gyoja} or \cite[Chapter 6 and Section 9.5]{kashi}.

Let $f\in \complex[X]$ be a non-zero polynomial. Then there is a differential operator $P(s)\in \dif_X[s]:=\dif_X\otimes_\complex \complex[s]$ and a non-zero polynomial $b(s)\in \complex[s]$ such that 
$$P(s)\cdot f^{s+1}(x) = b(s)\cdot f^s(x).$$
We call such a function $b_f(s)$ the $b$-\textit{function} of $f$ if it is monic of minimal degree. All roots of $b_f(s)$ are negative rational numbers \cite[Theorem 6.9]{kashi}.

For any $r\in\complex$, we can consider the $\dif_X$-module $\dif_X f^r$ that is the $\dif_X$-submodule generated by $f^r$ of the $\dif_X$-module $\complex[X]_f \cdot f^r$ consisting of (multi-valued) functions of the form $a f^r$, where $a\in \complex[X]_f$. Then $\dif_X f^r$ is regular holonomic (this follows as in \cite[Lemma 2.8.6]{gyoja}). Clearly, $\dif_X f^{r+1} \subseteq \dif_X f^r$, and when $r$ is not a root of $b_f(s)$ then equality holds. Equality may hold even when $r$ is a root,  as shown in \cite{saito}. Moreover, if none of $r,r+1,r+2,\dots$ is a root of $b_f(s)$, then $\dif_X f^r$ is irreducible by \cite[Corollary 6.25]{kashi}. For more about the modules $\dif_X f^r$, see \cite{uli} for a survey.

We call a non-zero polynomial $f\in \complex[X]$ a \itshape
semi-invariant\normalfont, if there is an algebraic character
$\sigma\in \Hom(G,\complex^*)$ such that $g\cdot f = \sigma(g)
f$ (that is, $f(gx)=\sigma(g)^{-1}f(x)$) and in this case we call $\sigma$ the weight of $f$.

The $\dif_X$-module $\dif_X f^{r}$ is $G$-equivariant if and only if $\sigma^r$ is an algebraic character of $G$.

Since $X$ is prehomogeneous, the multiplicities
$m_{\sigma^n}(\complex[X]) = 1$, for all $n\in\nat$ (see
\cite[\S 4. Proposition 3]{saki}). We have a non-zero dual semi-invariant
\[
f^*(\partial)\in \complex[X^*]
\]
of weight $\sigma^{-1}$, which we view as a differential operator. We have (see \cite[Corollary 2.5.10]{gyoja}):

\begin{theorem}[{\cite[Corollary 2.5.10]{gyoja}}]
Let $f$ be semi-invariant and take $f^*$ as above. Then we have the following equation
\begin{equation}\label{eq:good}
f^*(\partial)f(x)^{s+1}=b(s)f(x)^s,
\end{equation}
and $b(s)$ coincides with the $b$-function $b_f(s)$ of $f$ up to a non-zero constant factor. Moreover, $\deg b_f(s) = \deg f$.
\end{theorem}

We have the following result (which holds more generally for semi-invariants with multiplicity-free weights, as defined in \cite[Definition 1.1]{bub2}):

\begin{proposition}\label{prop:root}
Assume $X$ is prehomogeneous and $f\in \complex[X]$ is an irreducible semi-invariant. Then the $\dif_X$-module $\dif_X f^{r}/\dif_X f^{r+1}$ is non-zero if and only if $r$ is a root of the $b$-function $b_f(s)$, in which case it has a unique irreducible quotient $L_f^r$. Moreover, if $r_1,r_2,\dots,r_d$ denote the distinct roots of $b_f(s)$, then the irreducibles $L_f^{r_k}$ are pairwise non-isomorphic, for $k=1,\dots,d$. 
\end{proposition}

\begin{proof}
Fix $r=r_k$ and let $\sigma$ be the weight of $f$. First, we show that $f^{r}\notin \dif_X f^{r+1}$. Assume the contrary, \emph{i.e.} 
$$Q\cdot f^{r+1} = f^{r},$$
for some $Q\in \dif_X$. Decompose $Q$ into $\lie$-isotypical components $Q=\ds_{i=1}^l Q_j$, with $Q_j\in(\dif_X)_{(\lambda_j)}$, for some pairwise different dominant weights $\lambda_1,\dots,\lambda_l$. Then $Q_j \cdot  f^{r+1}$ lies in the isotypical component of weight $\sigma^{r+1}\lambda_j$. Hence, $Q_j$ annihilates $f^{r+1}$ unless $\lambda_j=\sigma^{-1}$. So we can assume without loss of generality that $Q$ is a semi-invariant differential operator of weight $\sigma^{-1}$. As in equation \eqref{eq:good}, this implies an equation
\[
Q\cdot f^{s+1} = b'(s)f^s,
\]
for some polynomial $b'(s)$. Since $b_f(s)$ is minimal, we have $b_f(s)|b'(s)$. This implies that $Q\cdot f^{r+1}=0$, a contradiction. Hence $\dif_X f^{r}/\dif_X f^{r+1}$ is non-zero.

Since $\complex[X]_f \cdot f^r$ is a semi-simple $\lie$-module,  $\dif_X f^{r}/\dif_X f^{r+1}$ is semi-simple as well, and using the argument in Lemma \ref{lem:irrind} (with $v_\lambda = f^r$) we get that $\dif_X f^{r}/\dif_X f^{r+1}$ has a unique irreducible quotient, say $L_f^r$.

Lastly, if we look at semi-invariants of the action of $\mathfrak{g}$ on $L_f^{r_k}$, we see that $f^{r_k}$ has weight $\sigma^{r_k}$, and $L_f^{r_k}$ has no semi-invariants with a weight of a larger power of $\sigma$. Hence the $\dif_X$-modules $L_f^{r_k}$ are pairwise non-isomorphic.
\end{proof}

\begin{remark}
  If $U$ is an open subset in a vector space such that its complement has codimension $\geq 2$, then $U$ is simply connected (see \cite[Chapter 4]{dimca}). Hence, in that case $\complex[U]$ is the only irreducible integrable connection on $U$.
\end{remark}

\begin{lemma}\label{lem:generic}
Let $f$ be an irreducible homogeneous polynomial on $X$ such that the hypersurface $Z=f^{-1}(0)$ is normal, and denote the complement by $U=X\setminus Z$. Then $\pi_1(U)=\integ$ and the irreducible integrable connections on $U$ are (up to isomorphism) all of the form $\complex[U] \cdot f^r$, for $r\in \CC/\integ$.
\end{lemma}

\begin{proof}
By the Riemann-Hilbert correspondence, it is enough to see that the fundamental group of $U$ is $\pi_1(U) = \integ$. As proved by Kyoji Saito and L\^ e D\~ung Tr\'ang, the complements of divisors that are normal crossing in codimension $1$ have Abelian fundamental groups \cite{SaitoLe}. In particular, this applies to normal homogeneous hypersurfaces. Then \cite[Cor.~4.1.4]{dimca} gives $\pi_1(U)=\integ$.
\end{proof}

\begin{remark}\label{rem:spherical}
 By a result of T. Vust and H. Kraft, if $f$ is an irreducible semi-invariant of weight $\sigma$ and $Z=f^{-1}(0)$ is $G$-spherical, then $Z$ is normal with rational singularities, see \cite[Theorem 5.1]{pany}. 

The connection $\complex[U] \cdot f^r$ is equivariant if and only if $\sigma^r$ is an algebraic character of $G$. In particular, writing $\sigma = \chi^d$ for some character $\chi: G\to \complex^*$ with $d\in\integ_{>0}$ maximal, the irreducible equivariant integrable connections on $U$ are $\complex[U] \cdot f^{i/d}$, for $i=0,\dots, d-1$. In case $U$ is a $G$-orbit, say $U\cong G/H$, this gives an isomorphism $H/H^0 \cong \integ/d\integ$ induced by the restriction of $\chi$ to $H$.
\end{remark}

As usual, let $B$ be a Borel subgroup of $G$ containing a maximal torus $T$.

\begin{lemma}\label{lem:highest}
Let $G$ be a connected reductive group, $V_\lambda$ an irreducible $G$-module and $v\in V_\lambda$ a highest weight vector of weight $\lambda$. Take $\tilde{G}=G \times \CC^*$, where $\CC^*$ acts on $V_\lambda$ by scalars in the usual way. Then the stabilizer $\tilde{G}_v$ is connected. Moreover, the stabilizer $G_v$ is connected if and only if $\lambda$ is not a non-trivial power of a highest weight (\emph{i.e.} of a polynomial function on $T$). 
\end{lemma}

\begin{proof}
It is well-known that $G \cdot [v]$ is the unique closed $G$-orbit in $\mathbb{P}(V)$, hence the stabilizer of $[v]$ in $\mathbb{P}(V_\lambda)$ is a parabolic subgroup $P$ of $G$. In particular, $P$ is connected. Clearly, $\tilde{G}_v \cong P$, and $G_v$ is the kernel of $\lambda: P \to \CC^*$. 
We have $P/G_v \cong \CC^*$, and since $P/G_v^0$ is 1-dimensional, we must have $P/G^0_v \cong \complex^*$ induced by a map $\chi: P \to \CC^*$. The map $P/G_v^0 \to P/G_v$ is then equivalent to a map $\complex^* \to \complex^*$, which must be a $k$-th power map, for some $k\in\integ$. Then $\lambda = \chi^k$. Hence, $H=H^0$ if and only if $k\in\{\pm 1\} $ if and only if $\lambda$ is not a non-trivial power.
\end{proof}

\subsection{Twisted Fourier transform}\label{subsec:four}

Let $X=\CC^n$ be an affine space with the linear action of a connected reductive group $G$. First, we introduce the Fourier transform of a $\dif$-module. Let $X^*$ denote the dual space of $X$. We denote $\mathcal F$ the Fourier automorphism
\begin{gather*}
\mathcal{F}\colon \dif_X\to \dif_{X^*}
\end{gather*}
for which  $\mathcal{F}(x_i)=\partial_i$ and $\mathcal{F}(\partial_i)=-x_i$ for  $i=1,\dots, n$.
It induces an involution  $M \mapsto \mathcal{F}(M)$ on
$\dif_X$-modules that
preserves $G$-equivariance, and there
is an isomorphism of $G$-modules 
\[
\mathcal{F}(M)=M\otimes\det(X)^{-1},
\]
where $\det(X)=\bigwedge^n X$ is the character of $G$ induced by the composition $G\to \GL(X) \stackrel{\det}{\longrightarrow} \complex^*$. So $\mathcal{F}$ is an involutive (covariant) functor giving equivalences of categories
\[\mathcal{F} : \op{Mod}_G(\dif_{X}) \xrightarrow{\sim} \op{Mod}_G(\dif_{X^*})\,\, \text{ and }\,\, \mathcal{F} : \op{mod}_G(\dif_X) \xrightarrow{\sim} \op{mod}_G(\dif_{X^*}).\]

Now let $T\subset G$ be a maximal torus of $G$ and $B\subset G$ a Borel subgroup containing $T$. It is well-known that there is an involution $\theta \in \op{Aut}(G)$ such that $\theta(t)=t^{-1}$ for all $t\in T$ and $B\cap \theta(B)=T$ (see \cite[Section 1.2]{pany}). If $V$ is any $G$-module, we can twist the action of $G$ by $\theta$ to obtain a $G$-module $V^*$, which is isomorphic to the usual dual representation of $V$. 

Twisting the action of $G$ on $X$ by $\theta$ gives another functor $\op{Mod}_G(\dif_{X}) \xrightarrow{\sim} \op{Mod}_G(\dif_{X^*})$ that sends an equivariant $\dif_X$-module $M$ to an equivariant $\dif_{X^*}$-module $M^*$. Moreover, if $M$ has a decomposition $M\cong \oplus M_i$ into irreducibles as a $G$-module, then $M^*$ decomposes as $M^*\cong \oplus M_i^*$. Composing this functor with the Fourier transform $\mathcal{F} : \op{Mod}_G(\dif_{X^*}) \xrightarrow{\sim} \op{Mod}_G(\dif_{X})$ we obtain an involutive (covariant) functor $\tilde{\mathcal{F}}$ giving self-equivalences
\[\tilde{\mathcal{F}} : \op{Mod}_G(\dif_{X}) \xrightarrow{\sim} \op{Mod}_G(\dif_{X})\,\, \text{ and }\,\, \tilde{\mathcal{F}} : \op{mod}_G(\dif_X) \xrightarrow{\sim} \op{mod}_G(\dif_{X}).\]
We call $\tilde{\mathcal{F}}$  the \textit{twisted Fourier transform}. By the above, if $M\in \op{Mod}_G(\dif_{X})$ has a $G$-decomposition $M\cong \bigoplus M_i$, then we have a $G$-decomposition (see also \cite[Section 2.5]{claudiu1})
\begin{equation}\label{eq:fourier}
\tilde{\mathcal{F}}(M)\cong \bigoplus (M_i^* \otimes \det(X)).
\end{equation}

Now assume for the rest of the section that $G$ acts on $X$ with finitely many orbits.  Then each $G$-orbit is conic, i.e. $\CC^*$-stable. In fact, the following discussion holds more generally for conic subvarieties.

We recall a relationship between the characteristic cycles $\charC(M)$ and $\charC(\tilde{\mathcal{F}}(M))$. For a $G$-orbit $O \subset X$, there exists a $G$-orbit in $O^* \subset X^*$ (the projective dual of $O$) such that under the natural identification $T^*X \cong T^* X^*$ we have (see \cite{pya},\cite[Section 2]{tevelev})
\[\ove{T^*_O X} \cong \ove{T^*_{O^*} X^*}.\]
This establishes a bijection between the $G$-orbits of $X$ and the $G$-orbits of $X^*$ (for $G$ not necessarily reductive), called the \emph{Pyasetskii pairing}.

The automorphism $\theta$ gives another such bijection between the orbits of $X$ and $X^*$. Composing the two, we obtain an involution (also referred to as Pyasetskii pairing) on the $G$-orbits of $X$, which we will denote by  $O \to O^{\vee}$.

By Theorem  \ref{thm:eqfin}, for $M \in \op{mod}_G(\dif_X)$ there exist $G$-orbits $O_i$ and positive integers $m_i$ ($1 \leq i \leq r$) such that 
\[\charC(M) = \ds_{i=1}^r m_i \cdot [\ove{T^*_{O_i} X}].\]
Then we have (see \cite[Theorem 3.2]{hotkash})
\begin{equation} \label{eq:four-char}
\charC(\tilde{\mathcal{F}}(M)) = \ds_{i=1}^r m_i \cdot [\ove{T^*_{O_i^{\vee}} X}].
\end{equation}

Since in general it does not preserve inclusions of orbit closures (so it is not an automorphism of the corresponding Hasse diagram), the Pyasetskii pairing is difficult to describe explicitly (for examples, see \cite[Chapter 2]{tevelev}). On the other hand, the Fourier transform is an (involutive) automorphism of the quiver $\widehat{Q}$ corresponding to the category $\op{mod}_G(\dif_X)$. As can be seen in the next section, often one can obtain implicitly the Pyasetskii pairing by determining the Fourier involution on the quiver $\widehat{Q}$. 

Although defined only for an affine space, the (equivariant) twisted Fourier transform can be lifted for a smooth affine variety using Proposition \ref{prop:redaff}

\section{Categories of equivariant $\dif$-modules for irreducible\\ spherical vector spaces}
\label{sec:examples}

Throughout in this section let $X$ be a vector space which is spherical with respect to the linear action of a connected reductive group $G$; such a space $X$ is called a \textit{multiplicity-free} $G$-space in the literature (see \cite{howumed}). To avoid possible confusion with Definition \ref{def:multfree}, we will call such $X$ a $G$-\textit{spherical vector space}. In this section we describe the categories of $G$-equivariant coherent $\dif$-modules using quivers with relations $\widehat{Q}$ as in Theorem \ref{thm:findim} for the $G$-spherical vector spaces that are irreducible representations. Such spaces were classified (up to geometric equivalence) by Kac \cite[Theorem 3.3]{kac}, and include the spaces of $m\times n$ matrices, skew-symmetric matrices and symmetric matrices. 

We say two rational representations $\rho_1: G_1 \to \GL(V_1)$ and $\rho_2: G_2 \to \GL(V_2)$ are \textit{geometrically equivalent}, if $\rho_1(G_1)$ coincides with $\rho_2(G_2)$ under an isomorphism $V_1 \to V_2$. For example, any representation is geometrically equivalent to its dual representation (see Section \ref{subsec:four}).

Now we recall the above-mentioned list, and refer to \cite{kac} for the notation.

\begin{theorem}\label{thm:list}

Up to geometric equivalence, a complete list of irreducible spherical vector spaces of connected reductive groups is as follows:

\begin{enumerate}
\item $\GL_n\otimes \SL_n\,$, $\SL_m \otimes \SL_n$ for $m\neq n$, $\Sym_2 \GL_n\, , \bigwedge^2 \SL_n$ for $n$ odd, $\bigwedge^2 \GL_n$ for $n$ even,\smallskip\\
$\SO_n \otimes \CC^*$ for $n\geq 3\,$, $\spin_7 \otimes \CC^*\,$, $\spin_9 \otimes \CC^*\,$, $\spin_{10}\,, \Gtwo \otimes \CC^*\,$, $\E\otimes \CC^*\,, \Sp_{2n}\,$,\smallskip\\
$ \Sp_{2n} \otimes \GL_2\,$ , $\Sp_{2n} \otimes \SL_3 \,$ , $\Sp_{4} \otimes \GL_4 \,$ , $\Sp_{4} \otimes \SL_m$ for $m > 4$.

\item $G\otimes \CC^*$ for the semi-simple groups $G$ from list 1.
\end{enumerate}

\end{theorem}

The following is our main result.

\begin{theorem}\label{thm:main}
Let $G$ be a connected reductive group and $X$ an irreducible $G$-spherical vector space. Then $\op{mod}_G(\dif_X) \cong \rep(\widehat{Q})$, where $\widehat{Q}$ is a quiver given (up to some isolated vertices) by:
\begin{itemize}
\item[(a)] The quiver $\AAA_3^c$ as in \eqref{eq:doubb}, if $G\to \GL(X)$ is geometrically equivalent to $\Sp_4 \otimes \GL_3\,$, or 
\item[(b)] The quiver $\EEE$ as in \eqref{eq:e6}, if $G \to \GL(X)$ is geometrically equivalent to $\Sp_{4} \otimes \GL_4 \,$, or 
\item[(c)] A disjoint union of quivers of type $\AAA$ as in \eqref{eq:double}, otherwise. 
\end{itemize}
In particular, if $G \to \GL(X)$ is not equivalent to $\Sp_{4} \otimes \GL_4$,  then there are (up to isomorphism) only finitely many indecomposable $G$-equivariant coherent $\dif_X$-modules.
\end{theorem}

\vspace{0.1in}

In this section, we give a proof of this theorem by a case-by-case consideration according to Theorem \ref{thm:list}. We note that the classification above is only up to geometric equivalence, and by Lemma \ref{lem:faithful} the categories of equivariant $\dif_X$-modules are not necessarily the same for two geometrically equivalent representations. Hence, in each case $\rho : G \to \GL(X)$ in Theorem \ref{thm:list} we mention how the categories can change with the choice of a different group $\tilde{G} \to \rho(G)$. By Lemma \ref{lem:faithful}, we can assume $\tilde{G}$ to be a covering group of $\rho(G)$. We will see that for each of these cases the quiver of $\op{mod}_{\tilde{G}}(\dif_X)$ differs only by some isolated vertices from a union of connected components of the quiver of $\op{mod}_G(\dif_X)$. To this end, we will compute most of the fundamental groups of the orbits and use Proposition \ref{prop:semiorbit} below.

For some of these cases, the quivers corresponding to $\op{mod}_\Lambda^{rh} (\dif_X)$ were described in \cite{bragri,nang4,nang,nang3,nang5,nang2}.
We note however that the description of the category $\op{mod}_\Lambda^{rh} (\dif_X)$ does not give immediately a description of the category $\op{mod}_{G}(\dif_X)$, since the latter in general is not closed under extensions (see Proposition \ref{prop:serre}). However, we have the following result, where  $(G,G)$ denotes the semi-simple part of $G$.

\begin{proposition}\label{prop:semiorbit}
Let $X$ be one of the irreducible $G$-spherical vector spaces in Theorem \ref{thm:list} distinct from $\spin_9 \otimes \CC^*$ and $\Gtwo \otimes \CC^*$. If there are no non-constant $G$-semi-invariants in $\complex[X]$, then each $G$-orbit is a $(G,G)$-orbit, and $\op{mod}_{(G,G)}(\dif_X) = \op{mod}_\Lambda^{rh} (\dif_X)$. Otherwise, let $Z$ denote the hypersurface defined by a non-constant $G$-semi-invariant. Then each $G$-orbit in $Z$ is a $(G,G)$-orbit, and $\op{mod}_{(G,G)}^Z(\dif_X)=\op{mod}_\Lambda^{rh, Z} (\dif_X)$.
\end{proposition}

\begin{proof}
The claim about $G$-orbits that are also $(G,G)$-orbits follows by \cite[Proposition 3.3]{kac}. If $(G,G)$ is also simply connected, the claim about the categories follows by Lemma \ref{lem:simplyconn}. The only case when $(G,G)$ is not simply connected is $\SO_n \otimes \CC^*$. However, in this case $Z$ has only one non-zero orbit, which is the orbit of the highest weight vector. By working with the group $\spin_n$, we see using Lemma \ref{lem:highest} that the stabilizer is connected, hence the orbit is simply connected.
\end{proof}

In general, $\op{mod}_G^{Z}(\dif_X)$ is a subcategory of $\op{mod}_{(G,G)}^{Z}(\dif_X)$, but in our cases we will see that often equality holds.

The quivers that we describe resemble to some extent the holonomy diagrams obtained in \cite{kimu}. The reduction techniques we use for the spaces of matrices are similar to the slice methods used in \cite{bub2} to compute the $b$-functions of their semi-invariants.

\subsection{The space of $m\times n$ matrices}\label{subsec:gener}

Let $X=X_{m,n}$ be the space of $m\times n$ matrices together with the action of $G=G_{m,n}:=\GL_m(\CC)\times \GL_n(\CC)$ defined  by $(g_1,g_2)\cdot M = g_1  M g_2^{-1}$, where $M\in X$ and $(g_1,g_2)\in G$. Assume without loss of generality that $m\geq n$. Then we have $n+1$ orbits $O_0, O_1,\dots,O_n$, where $O_i$ is the subset of matrices of rank $i$. Moreover, each orbit has a connected stabilizer. By Theorem \ref{thm:eqfin} we have $n+1$ irreducible equivariant $\dif_X$-modules, all fixed under duality $\dual$. We denote the corresponding vertices of the quiver $\AAA_{n+1}$ by $(0),(1),\dots, (n)$ (from left to right), where $\AAA_{n+1}$ the obvious extension of the quiver  in \eqref{eq:double} with additional vertex $(0)$.

\begin{theorem}\label{thm:gener}
Take $X=X_{m,n}$ and $G=\GL_m(\CC)\times \GL_n(\CC)$ as above. Then we have:
\begin{itemize}
\item[(a)]  If $m=n$, the category $\op{mod}_G(\dif_X)=\op{mod}_{\GL_n \times \SL_n}(\dif_X)$ is equivalent to $\rep(\AAA_{n+1})$.

\item[(b)] If $m\neq n$, the categories $\op{mod}_{\GL_m\times \GL_n}(\dif_X)=\op{mod}_{\SL_m\times \SL_n}(\dif_X) = \op{mod}_{\Lambda}^{rh}(\dif_X)$  are semi-simple.
\end{itemize}
\end{theorem}

\begin{proof}
We prove the result by induction on $n$. For $n=0$ (\emph{i.e.} $X=\{0\}$) the result is clear. Now take any $n>0$. Recall the full category $\op{mod}_G(\dif_{X})_0$ as in Lemma \ref{lem:restrict} with $Y=\{0\}$. First, we construct a functor
\begin{gather}\label{functor-F}
  F\colon \op{mod}_{G_{m,n}}(\dif_{X_{m,n}})_0 \to
  \op{mod}_{G_{m-1,n-1}}(\dif_{X_{m-1,n-1}}).
\end{gather}
Let $X_{x_{11}}$ denote the principal open set $x_{11}\neq 0$. This is an $H$-stable subset, where $H$ is the subgroup of $G$ of matrices of the form
$$H=\begin{bmatrix} 
x & 0\\
v & A
\end{bmatrix} \times
\begin{bmatrix} 
y & w\\
0 & B
\end{bmatrix},$$
where $x,y \in \complex^*$, $v$ (resp.\ $w$) is a column (resp.\ row) vector in $\complex^{m-1}$ (resp.\ in $\complex^{n-1}$), and $(A,B)\in G_{m-1,n-1}$.
By Lemma \ref{lem:restrict}, the restriction $j^*: \op{mod}_G(\dif_{X})_0\to  \op{mod}_H(\dif_{X_{x_{11}}})$ is an embedding of categories. 

We embed $X_{m-1,n-1}$ into $X_{x_{11}}$ by placing an $(m-1) \times (n-1)$ matrix $M$ as $\begin{bmatrix} 1 & 0 \\ 0 & M\end{bmatrix}.$

Let $G'_{m-1,n-1}= G_{m-1,n-1} \times \complex^*$, which we embed into $H$ by placing a tuple $(A,B,c)$ as 
$$\left(\begin{bmatrix} c & 0 \\ 0 & A \end{bmatrix},\begin{bmatrix} c & 0 \\ 0 & B \end{bmatrix}\right).$$
Note that the factor $\complex^*$ is acting trivially on $X_{m-1,n-1}$, and since $\CC^*$ is connected, we have $\op{mod}_{G_{m-1,n-1}}(\dif_X) \cong \op{mod}_{G'_{m-1,n-1}}(\dif_X)$. It is easy to see that the natural map 
\[H\times_{G'_{m-1,n-1}} X_{m-1,n-1}\to X_{x_{11}}\]
is an isomorphism of $H$-varieties. Hence we define $F$ in (\ref{functor-F}) to be the composition of $j^*$ with the isomorphism from Proposition \ref{prop:induct}. In particular, $F$ is an embedding of categories. 

Now take the open orbit $U=O_n$ with the inclusion $j': U \to X$. The category $\op{mod}_{G}(\dif_{U})$ is equivalent to that of vector spaces since the stabilizer is connected. In particular, $\orb_{U}$ is a simple injective in $\op{mod}_{G}(\dif_{U})$, and by adjunction, we see that $j'_*\orb_{U}$ is injective and indecomposable (but perhaps not simple) in $\op{mod}_{G}(\dif_{X})$. Now we discuss according to the two cases of the theorem: square and non-square.

(a) $j'_* \orb_{U}= \complex[X]_f$, where $f$ denotes the determinant. Since the roots of the $b$-function of $f$ are $-1,-2,\dots, -n$ (see \cite[Section 2.1]{kimu}), we have $n+1$ pair-wise non-isomorphic irreducible $\dif_X$-modules $\complex[X], L^{-1}_f, \dots, L^{-n}_f$ by Proposition \ref{prop:root}. Hence this is a complete list of all irreducibles (we will see that the simple $L^{-k}_f$ corresponds to quiver vertex $(n-k)$). Since $\complex[X]_f$ is multiplicity-free (since $U$ is spherical), we have in fact $L_f^r \cong \dif_X f^r/ \dif_X f^{r+1}$ for $r=-1,-2,\dots, -n$. Put $L_f^0 = \complex[X]$. For all $r=0,-1,\dots, -n+1$, we have exact sequences 
\[ 0\to L_f^r \to \dif_X f^{r-1}/\dif_X f^{r+1} \to L_f^{r-1}\to 0.\]
Applying Lemma \ref{lem:irrind} to $ \dif_X f^{r-1}/\dif_X f^{r+1}$, we see that these exact sequences do not split.

Since $\complex[X]_f$ is the injective envelope of $\complex[X]$, its composition series as above gives us a path $(0)\xrightarrow{\alpha_1} 
(1) \xrightarrow{\alpha_2} (2) \xrightarrow{\alpha_3}\dots\xrightarrow{\alpha_{n-1}} (n-1) \xrightarrow{\alpha_n} (n)$ in the quiver $\widehat{Q}$ of $\op{mod}_G(\dif_X)$ with no relation involving only the arrows $\{\alpha_i\}$. By duality $\dual$, we must have an opposite path as well. Hence the quiver $\AAA_{n+1}$ must be a subquiver of $Q$ (at least on the level of graphs).

Note that since $\complex[X]$ appears in the injective $\complex[X]_f$ as a composition factor only once, we must have a relation $\alpha_n
\beta_n=0$, hence also $\beta_n \alpha_n = 0$. Moreover, there are no more arrows starting or ending in $(n)$. Applying a twisted Fourier
transform, $\tilde{\mathcal{F}}(\complex[X]) = L_f^{-n}$. Proceeding similarly at the vertex $(0)$, one deduces $\alpha_1 \beta_1 = 0 = \beta_1 \alpha_1$ and there are no other arrows in $Q$ starting or ending in $(0)$. 

Let $\widehat{Q}'$ be the quiver obtained from $\widehat{Q}$ by erasing vertex $(0)$ together with the arrows $\alpha_1,\beta_1$. Then $\widehat{Q}'$ corresponds to the category $\op{mod}_{G_{m,n}}(\dif_{X_{m,n}})_0$. By induction, we know that the quiver $\widehat{Q}_{n-1}$ of $\op{mod}_{G_{m-1,n-1}}(\dif_{X_{m-1,n-1}})$ is
\[
Q_{n-1} : \xymatrix{
(0) \ar@<0.5ex>[r]^{\alpha_1} & \ar@<0.5ex>[l]^{\beta_1} (1)
  \ar@<0.5ex>[r]^{\alpha_2} & \ar@<0.5ex>[l]^{\beta_2} \dots
  \ar@<0.5ex>[r]^{\alpha_{n-1}} &\ar@<0.5ex>[l]^{\beta_{n-1}} (n-1)},
\]
with relations $\alpha_i \beta_i = 0 = \beta_i \alpha_i$ for all $i=1,\dots,n-1$.
Since the functor $F$ in \eqref{functor-F} is an embedding, $\widehat{Q}'$ is a subquiver of the quiver $\widehat{Q}_{n-1}$ (see Remark \ref{rem:subquiv}). This implies $\widehat{Q}'=\widehat{Q}_{n-1}$ (with vertex $(i)$ of $Q'$ corresponding to vertex $(i-1)$ of $Q_{n-1}$) and imposes the relations $\alpha_i \beta_i = 0 = \beta_i \alpha_i$ on $\widehat{Q}'$ as well. Since we have no relations involving only the arrows $\{\alpha_i\}$ (resp.\ $\{\beta_i\}$), there are no other relations on $\widehat{Q}'$. In particular, $F$ is an equivalence of categories. 

Since $\widehat{Q}'$ was obtained from $\widehat{Q}$ by erasing $(0)$, and we know $\alpha_1$ and $\beta_1$ are the only arrows of $(0)$ with no other relations involved, the inductive step is complete.  The entire argument works by replacing the group $G$ with the group $\GL_n \times \SL_n$.

Now let $\rho(G)$ be the image of the action map $\rho: G \to \GL(X)$ (which coincides with the image of the action map $\GL_n\times \SL_n \to \GL(X)$). By Lemma \ref{lem:faithful}, we have $\op{mod}_{\rho(G)} (\dif_X) = \op{mod}_{G} (\dif_X)$. Since each non-open $G$-orbit is also a $(G,G)=\SL_n\times \SL_n$-orbit, it follows that each non-open orbit is simply connected, since $\SL_n$ is so and the $(G,G)$-stabilizers are connected. Let $\tilde{G} \to \rho(G)$ be a covering group of $\rho(G)$. If $M$ is a $\tilde{G}$-equivariant $\dif_X$-module that is not $G$-equivariant, then it must come from an integrable connection on the open orbit $U$, hence by Lemma \ref{lem:generic} it corresponds to $\complex[U] f^r$, for some $r$ that is not an integer. 
By Lemma \ref{lem:faithful}, this shows that the quivers of $\op{mod}_{G} (\dif_X)$ and $\op{mod}_{\tilde{G}} (\dif_X)$ coincide up to some isolated vertices in the latter.

(b) Here we have $j'_* \orb_{U} = \complex[X]$. This implies that $\complex[X]$ is a simple injective $\dif_X$-module, and  by duality also projective. Thus, $(n)$ is an isolated vertex of the quiver, and by the twisted Fourier transform, $(0)$ is isolated as well. Deleting vertex $(0)$ we get the quiver of $\op{mod}_{G_{m,n}}(\dif_{X_{m,n}})_0$ that embeds via $F$ into $\op{mod}_{G_{m-1,n-1}}(\dif_{X_{m-1,n-1}})$. By induction, the latter is a semi-simple category, hence so is $\op{mod}_{G_{m,n}}(\dif_{X_{m,n}})_0$ and $\op{mod}_G(\dif_X)$ as well. We note that all orbits in this case are also $\SL_{m}\times \SL_{n}$-orbits (with trivial stabilizers), and the entire argument above works, \emph{mutatis mutandis}, replacing $G$ with this group. Hence $\op{mod}_G(\dif_X)=\op{mod}_{\SL_m\times \SL_n}(\dif_X)$. By Proposition \ref{prop:semiorbit}, this category is coincides with $\op{mod}_{\Lambda}^{rh}(\dif_X)$.
\end{proof}

\begin{remark}
We can avoid the reduction method from Section \ref{subsec:reduc} in the proof of the above theorem. In part (a) one can use Corollary \ref{cor:quivnice} instead of the argument by induction, while part (b) follows also from the fact that intersections of irreducible components of $\Lambda$ have codimension $\geq 2$ (see \cite[Proposition 2.10]{strick}), together with \cite[Theorem 6.7]{macvil} (see also \cite{kk}).
\end{remark}

\subsection{The space of skew-symmetric matrices}\label{subsec:skew}

Now let $X=\bigwedge^2 \complex^n$ be the space of $n\times n$
skew-symmetric matrices with the action of $G=\GL_n(\CC)$ defined by $g \cdot M = g  M  g^t$, where $M\in X, g\in G$. Let $r=\lfloor n/2\rfloor$. There are $r+1$ orbits $O_0,O_1,\dots, O_r$, where $O_i$ is the set of skew-symmetric matrices of rank $2i$. All the stabilizers are connected. If $n$ is even, then there is a semi-invariant (the Pfaffian) and the roots of its $b$-function are $-1,-3,\dots,-(n-1)$ (see \cite[Section 2.3]{kimu}) that give $r+1$ simple $\dif_X$-modules (including $\complex[X]$) as in Proposition \ref{prop:root}. We have the following result, whose proof is analogous to Theorem \ref{thm:gener}:

\begin{theorem}
Take $X=\bigwedge^2 \complex^n$ and $G=\GL_n(\CC)$ as above. Then we have:
\begin{itemize}
\item[(a)]  If $n=2r$ is even, then the category $\op{mod}_G(\dif_X)$ is equivalent to $\rep(\AAA_{r+1})$.
\item[(b)] If $n=2r+1$ is odd, then the categories $\op{mod}_G(\dif_X)=\op{mod}_{\SL_n}(\dif_X) = \op{mod}_{\Lambda}^{rh}(\dif_X)$  are semi-simple.
\end{itemize}
\end{theorem}

\subsection{The space of symmetric matrices}\label{subsec:symm}

Now let $X=\Sym_2 \complex^n$ be the space of $n\times n$ symmetric matrices with the action of $G=\GL_n(\CC)$ defined by $g \cdot M = g  M  g^t$, where $M\in X, g\in G$. There are $n+1$ orbits $O_0,O_1,O_2,\dots, O_n$ given again by ranks. What makes this case more interesting than the previous ones is that the stabilizers of each non-zero orbit have $2$ connected components (although this can be avoided when $n$ is odd by choosing the group $\GL_n/\{\pm I_n\}$ instead, hence making the action faithful). The semi-invariant is the symmetric determinant $f=\det Y$, where $Y=Y^t$ is a generic symmetric $n\times n$ matrix of variables. The degree of $f$ is $n$ with $G$-character $\det^2$. Moreover, the roots of the $b$-function of $f$ are $-1,-3/2,-2,\dots,-(n+1)/2$ (see \cite[Section 2.2]{kimu} or \cite[Example 2.8]{bub2}). By Proposition \ref{prop:root}, the roots give $n+2$ simples
\[\complex[V], L^{-1}_f, L^{-2}_f \dots, L_f^{-\lfloor (n+1)/2 \rfloor}, \mbox{ and } \dif f^{-1/2}, L^{-3/2}_f, \dots , L^{-1/2-\lfloor n/2 \rfloor}_f,\]
out of a total $2n+1$ simple equivariant $\dif_X$-modules. From the proof of Theorem \ref{thm:symm}, it turns out that we have equality $\dif f^r/\dif f^{r+1}=L^r_f$ for all roots $r$, and these will correspond precisely to the non-isolated vertices of the quiver $\widehat{Q}$. The other $n-1$ simples will correspond to isolated vertices in $\widehat{Q}$ (\emph{i.e.} no arrows connected to them), moreover, they have no $\mathfrak{sl}_n(\CC)$-invariant (non-zero) sections. This was shown in \cite{claudiu1, claudiu5} by explicit computations, and so it provides a counterexample to a conjecture of T. Levasseur \cite[Conjecture 5.17]{leva}. We give a simple conceptual proof of this.

\begin{proposition}\label{prop:leva}
Up to isomorphism, there are precisely $n-1$ simple $G$-equivariant $\dif_X$-modules that have no (non-zero) $\mathfrak{sl}_n(\CC)$-invariant sections.
\end{proposition}

\begin{proof}
An $\mathfrak{sl}_n(\CC)$-invariant section of a $\GL_n(\CC)$-equivariant $\dif_X$-module is $\GL_n(\CC)$-semi-invariant of weight $\det^k$, for some $k\in \integ$. Using Proposition \ref{prop:root}, we see that for all the possible powers $k\in \integ$, there is exactly one simple $\dif_X$-module with a non-zero $\GL_n(\CC)$-semi-invariant section of weight $\det^k$ among the $n+2$ simples
\[\complex[V], L^{-1}_f, L^{-2}_f \dots, L_f^{-\lfloor (n+1)/2 \rfloor}, \mbox{ and } \dif f^{-1/2}, L^{-3/2}_f, \dots , L^{-1/2-\lfloor n/2 \rfloor}_f.\]
Since $\Sym_2 \complex^n$ is of Capelli type (see \cite[Section 15]{howumed}), we conclude by Corollary \ref{cor:nocommon} (c) that the other $n-1$ simples do not contain (non-zero) $\mathfrak{sl}_n(\CC)$-invariant sections.
\end{proof}

Now we proceed with the determination of the quiver $\widehat{Q}$, which has $2n+1$ vertices. We introduce the following notation. We label by vertex $(i)$ (resp.\ $(i)'$) the simple in $\op{mod}_G(\dif_X)$ with support $\ove{O_i}$ corresponding to the trivial (resp.\ non-trivial) $G$-equivariant simple local system on $O_i$, which in turn corresponds to the trivial (resp.\ sign) representation of the two-element group $\integ_2=\{\pm 1\}$. By convention, $(0) = (0)'$.  Let $\epsilon$ denote $0$ if $n$ is even and $1$ if $n$ is odd.

\begin{theorem}\label{thm:symm}
Take $X=\Sym_2 \complex^n$ and $G=\GL_n(\CC)$ with the notation as above. The category $\op{mod}_G(\dif_X)$ is equivalent to $\rep(\widehat{Q})$, where the following are connected components of $\widehat{Q}$
\[\xymatrix{
(1-\epsilon) \ar@<0.5ex>[r] & \ar@<0.5ex>[l]  (3-\epsilon)
\ar@<0.5ex>[r] & \ar@<0.5ex>[l] \dots
\ar@<0.5ex>[r] &\ar@<0.5ex>[l]  (n-3)
\ar@<0.5ex>[r] &\ar@<0.5ex>[l]  (n-1)
\ar@<0.5ex>[r] & \ar@<0.5ex>[l] (n)},\]
\[\hspace{1.5pc}\xymatrix{
(\epsilon)' \ar@<0.5ex>[r] & \ar@<0.5ex>[l]  (\epsilon+2)'
\ar@<0.5ex>[r] & \ar@<0.5ex>[l] \dots
\ar@<0.5ex>[r] &\ar@<0.5ex>[l]  (n-4)'
\ar@<0.5ex>[r] &\ar@<0.5ex>[l]  (n-2)'
\ar@<0.5ex>[r] & \ar@<0.5ex>[l] (n)'}, 
\]
with all $2$-cycles zero, and the other $n-1$ vertices of $\widehat{Q}$ are isolated.
\end{theorem}

\begin{proof}
The proof is similar to the proof of Theorem \ref{thm:gener} and we proceed by induction on $n$. The case $n=1$ holds, by inspection. Now take any $n>0$, and put $X_n:=X$, $G_n :=  G$. We construct a functor $F$ as in (\ref{functor-F})
\[
  F\colon \op{mod}_{G_{n}}(\dif_{X_{n}})_0 \to
  \op{mod}_{G_{n-1}\times \integ_2}(\dif_{X_{n-1}}).
\]
Let $X_{x_{11}}$ denote the principal open set $x_{11}\neq 0$. This is an $H$-stable subset, where $H$ is the subgroup of $G$ of matrices of the form $H=\begin{bmatrix} 
x & 0\\
v & A
\end{bmatrix}$,
where $x\in\complex^*$, $v$ is a column vector in $\complex^{n-1}$, and $A\in G_{n-1}$.
By Lemma \ref{lem:restrict}, the restriction $j^*: \op{mod}_G(\dif_{X})_0\to  \op{mod}_H(\dif_{X_{x_{11}}})$ is an embedding of categories. 

We embed $X_{n-1}$ into $X_{x_{11}}$ by placing an $(n-1) \times (n-1)$ symmetric matrix $M$ as $\begin{bmatrix} 1 & 0 \\ 0 & M\end{bmatrix}.$

Let $G'_{n-1}= G_{n-1} \times \integ_2$, which we embed into $H$ by placing a tuple $(A,\pm 1)$ as 
$\begin{bmatrix} \pm 1 & 0 \\ 0 & A \end{bmatrix}.$
Note that the factor $ \integ_2$ is acting trivially on $X_{n-1}$, hence 
\[\op{mod}_{G'_{n-1}}(\dif_{X_{n-1}})\cong \op{mod}_{G_{n-1}}(\dif_{X_{n-1}})\oplus \op{mod}_{G_{n-1}}(\dif_{X_{n-1}}),\]
by Lemma \ref{lem:faithful}.
The natural map 
\[H\times_{G'_{n-1}} X_{n-1}\to X_{x_{11}}\] 
is an isomorphism of $H$-varieties. Hence we let $F$ be the composition of $j^*$ with the isomorphism from Proposition \ref{prop:induct}. In particular, $F$ is an embedding of categories. 

By induction, we know that the quiver $\widehat{Q}_{n-1}$ of $\op{mod}_{G_{n-1}}(\dif_{X_{n-1}})$ is
\[\xymatrix{
(\epsilon) \ar@<0.5ex>[r] & \ar@<0.5ex>[l]  (\epsilon+2)
\ar@<0.5ex>[r] & \ar@<0.5ex>[l] \dots
\ar@<0.5ex>[r] &\ar@<0.5ex>[l]  (n-2)
\ar@<0.5ex>[r] & \ar@<0.5ex>[l] (n-1)},\]
\[\xymatrix{
(1-\epsilon)' \ar@<0.5ex>[r] & \ar@<0.5ex>[l]  (3-\epsilon)'
\ar@<0.5ex>[r] & \ar@<0.5ex>[l] \dots
\ar@<0.5ex>[r] &\ar@<0.5ex>[l]  (n-3)'
\ar@<0.5ex>[r] & \ar@<0.5ex>[l] (n-1)'}, 
\]
with all $2$-cycles zero. Hence, we can write $\op{mod}_{G'_{n-1}}(\dif_{X_{n-1}})\cong \op{mod}_H(\dif_{X_{x_{11}}}) \cong \rep(\widehat{Q}_{n-1}) \oplus \rep(\widehat{Q}_{n-1}^{-})$. Here $\widehat{Q}_{n-1}^{-}$ is just another copy of the quiver $\widehat{Q}_{n-1}$. The reason for this extra copy is that the component groups of the $H$-stabilizers of the orbits in $X_{x_{11}}$ are isomorphic to $\integ_2 \times \integ_2$. These component groups $\integ_2 \times \integ_2$ map naturally to the corresponding component groups $\integ_2$ of the $G$-stabilizers of the non-zero orbits in $X$ via $(a,b)\mapsto ab$, for $a,b\in \{\pm 1\}$. Under this map $\integ_2 \times \integ_2 \to \integ_2$, the trivial representation of $\integ_2$ induces (by \lq\lq restriction'') the trivial representation of $\integ_2\times \integ_2$, and the sign representation $\op{sgn}$ of $\integ_2$ induces the representation $\op{sgn} \otimes \op{sgn}$ of $\integ_2\times \integ_2$. This gives which of the simple $H$-equivariant $\dif_{X_{x_{11}}}$-modules are actually restrictions of $G$-equivariant $\dif_X$-modules.

Let $\widehat{Q}'$ be the quiver obtained from $\widehat{Q}$ by erasing vertex $(0)$ together with the arrows connected to it. Then $\widehat{Q}'$ corresponds to the category $\op{mod}_{G_{n}}(\dif_{X_{n}})_0$. Since the functor $F$ is an embedding, $\widehat{Q}'$ is a subquiver of the quiver $\widehat{Q}_{n-1} \cup \widehat{Q}_{n-1}^-$ (see Remark \ref{rem:subquiv}). By the discussion above regarding the component groups, the embedding of quivers is achieved as follows: for any $i=1,\dots,n$, the vertex $(i)$ of $Q'$ is sent to vertex $(i-1)$ of $\widehat{Q}_{n-1}$, while the vertex $(i)'$ of $\widehat{Q}'$ is sent to vertex $(i-1)'$ of $\widehat{Q}^{-}$. By erasing the other vertices of $\widehat{Q}_{n-1} \cup \widehat{Q}_{n-1}^-$ (that is, disregarding the $H$-equivariant $\dif_{X_{x_{11}}}$-modules that do not come from $G$-equivariant $\dif_X$-modules), we obtain that $\widehat{Q}'$ must be a subquiver of the following quiver (with all 2-cycles zero)

\[\xymatrix{
(1+\epsilon) \ar@<0.5ex>[r] & \ar@<0.5ex>[l]  (3+\epsilon)
\ar@<0.5ex>[r] & \ar@<0.5ex>[l] \dots
\ar@<0.5ex>[r] &\ar@<0.5ex>[l]  (n-1)
\ar@<0.5ex>[r] & \ar@<0.5ex>[l] (n)},\]
\[\hspace{0.2pc}\xymatrix{
(2-\epsilon)' \ar@<0.5ex>[r] & \ar@<0.5ex>[l]  (4-\epsilon)'
\ar@<0.5ex>[r] & \ar@<0.5ex>[l] \dots
\ar@<0.5ex>[r] &\ar@<0.5ex>[l]  (n-2)'
\ar@<0.5ex>[r] & \ar@<0.5ex>[l] (n)'}, 
\]
together with $n-1$ isolated vertices.

As in the proof of Theorem \ref{thm:gener}, due to the composition series of $\complex[X]_f$ and $\complex[X]_{f}\cdot f^{1/2}$ we can see that in the above quiver there are no other relations among the arrows. Applying the twisted Fourier transform, we see that $\tilde{\mathcal{F}}(\complex[X]) = L_f^{-(n+1)/2}$ is a composition factor in $\complex[X]_f$ (resp.\ $\complex[X]_{f}\cdot f^{1/2}$) when $n$ is even (resp.\ odd), which shows that the vertex $(0)$ attaches to the bottom (resp.\ top) component of $\widehat{Q}'$ in the desired way.

It is easy to see that the $\SL_n$-stabilizers of the non-open orbits have two connected components, hence their fundamental groups are equal to $\integ_2$. Moreover, by Proposition \ref{prop:semiorbit} we have $\op{mod}_{\GL_n}^Z(\dif_X)=\op{mod}_{\SL_n}^Z(\dif_X)=\op{mod}_\Lambda^{rh,Z}(\dif_X)$.   

Now if $\rho: G \to \GL(X)$ denotes the action map, we have $\rho(G) \cong G/\{\pm I_n\}$. When $n$ is even, $\op{mod}_{\rho(G)} (\dif_X)\cong \op{mod}_{G}(\dif_X)$, while when $n$ is odd, the quiver of $\op{mod}_{\rho(G)} (\dif_X)$ is only the top component of $\widehat{Q}$ (together with the respective isolated vertices).

Now let $\tilde{\rho}:\tilde{G} \to \rho(G)$ be a covering group with kernel $K$. When $n$ is even, then a simple $\tilde{G}$-equivariant $\dif_X$-module that is not $G$-equivariant must have full support, and correspond to an isolated vertex (see similar argument in the proof of Theorem \ref{thm:gener}). When $n$ is odd, note that $\pi_1(\rho(G)) = \integ$ and so $K=\integ / k \integ$ for some $k\in \nat$. If $k$ is even, then $\tilde{\rho}$ factors through $\rho$, and as before, a simple $\tilde{G}$-equivariant $\dif_X$-module that is not $G$-equivariant must correspond to an isolated vertex. If $k$ is odd, then all stabilizers of the non-open orbits must be connected, so again the quiver of $\op{mod}_{\tilde{G}} (\dif_X)$ can have only isolated vertices in addition to the quiver of $\op{mod}_{\rho(G)} (\dif_X)$.
\end{proof}

\subsection{The cases of type $\Sp_{2n} \otimes \GL_m$}\label{subsec:sp}

Here, we discuss the cases $\Sp_{2n}\otimes \GL_2 \, , \Sp_{2n} \otimes \GL_3\,$ (both cases with $n\geq 2$), and $\Sp_{4} \otimes \GL_m$ for $m\geq 4$.

First, let us recall some general facts about the representations of type $\Sp_{2n} \otimes \GL_m$. Here the group $G=\Sp_{2n} \times \GL_m$ acts on the space $X$ of $2n \times m$ matrices in the obvious way. In other words, the action is the restriction of the action from Section \ref{subsec:gener}, by restricting the group $\GL_{2n}$ to $\Sp_{2n}$. There are finitely many orbits, given as follows. Let $Y$ be a generic $2n\times m$ matrix with variable entries, and let $Y_0$ be any $2n\times m$ matrix. If $J$ is the $2n\times 2n$ skew-symmetric invertible matrix defining $\Sp_{2n}$ then each $G$-orbit is determined by two data (see \cite[Section 13]{howumed} or \cite{lovett}): rank and isometry type. These data are described by $r=\op{rank} Y_0$ and $s=\op{rank} Y_0^t J Y_0$, respectively. We will denote by $O_{r,s}$ the orbit corresponding to the data $(r,s)$. There are some restrictions on such pairs of non-negative integers: $2r-2n \leq s\leq r \leq m$, and $s$ must be even. We have $O_{r,s} \subset \ove{O}_{r',s'}$ if and only if both $r\leq r', \, s\leq s'$ hold. The codimension is given by \cite[Corollary 2.4]{lovett}
\begin{equation}\label{eq:codim}
\op{codim} O_{r,s} = (2n-r)(m-r)+ \frac{(r-s)(r-s-1)}{2}.
\end{equation}

There is a non-constant semi-invariant $f$ on $X$ if and only if $m\leq 2n$ and $m$ is even, in which case it is given by the Pfaffian of $Y^t J  Y$ (recall, $Y$ is generic). In this case, $f$ has degree $m$ and the roots of the $b$-function of $f$ are $-1, -3, \dots, -m+1$ and $-2n, -2n+2, \dots, -2n+m-2$ (see \cite[Proposition 3.1]{kimu} or \cite[Example 2.10]{bub2}). In particular, the hypersurface $f=0$ is normal by \cite[Theorem 0.4]{saito2} and we can apply Lemma \ref{lem:generic}. This implies that the stabilizer of the open $G$-orbit is always connected.

\begin{lemma}\label{lem:compint}
Let $m\leq n$, and denote by $f_{ij} \,\, (1\leq i < j \leq m)$ the entries of the upper triangular part of the $m\times m$ skew-symmetric matrix $Y^t J Y$ where $Y$ is a $2n\times m$ matrix of indeterminates. Then the orbit closure $\ove{O}_{m,0}$ is a (set-theoretic) complete intersection defined by $\{f_{ij}\}_{1\leq i < j \leq m}$. Moreover, they generate the ring of invariants
\[\complex[X]^{\Sp_{2n}} = \complex [f_{ij}]_{1\leq i<j \leq m}.\]
\end{lemma}

\begin{proof}
The first part follows since $m(m-1)/2=\op{codim} O_{m,0}$ by \eqref{eq:codim}. The second part is the First Fundamental Theorem for $\Sp_{2n}$.
\end{proof}

The following result will be sufficient to establish the simply connectedness of most of the encountered $G$-orbits.

\begin{lemma}\label{lem:simply}
Let $(r,s)$ be with either $s=0$ or $s=r <m$. Then the orbit $O_{r,s}$ is simply connected.
\end{lemma}

\begin{proof}
We identify $X$ with the space of linear maps $\Hom(E,F)$, where $\dim E= m$ and $\dim F=2n$. Let $\omega$ be the symplectic form on $F$.

First, consider the case $s=0$. Then we must have $r\leq n$. Notice that the image of an element $\phi \in O_{r,0}$ is isotropic with respect to $\omega$, hence an element of the isotropic Grassmannian $\op{IGr}(r,F)$ (of all isotropic subspaces). This gives a fiber bundle
\begin{equation}\label{eqn-bundle-IGr}
  Z \to O_{r,0} \to \op{IGr}(r,F),
\end{equation}
where the fiber $Z$ can be identified with maps in $\Hom(E, \complex^r)$ of rank $r$. The space $\op{IGr}(r,F)$ is simply connected, as (like $\op{Gr}(r,F)$) it is the quotient of a simply-connected group by a connected group.  If $r<m$, then $Z$ is simply connected too, in which case the exact sequence of homotopy groups attached to the fibration gives that $O_{r,0}$ is simply connected as well. So we can assume $r=m$, where $\pi_1(Z)=\integ$. Then take $O_r \subseteq X$ to be the maps of rank $r$. Since $m=r\leq n <2n$, $O_r$ is simply connected. The fiber bundle \eqref{eqn-bundle-IGr} can be realized as the pull-back of the fiber bundle $Z \to O_r \to \op{Gr}(r,F)$ on the Grassmannian. Note that the natural map $\pi_2(\op{IGr}(r,F))\to \pi_2(\op{Gr}(r,F))$ is an isomorphism, identifying both with $\integ$; via the Hurewicz theorem this reduces to checking that the appropriate Schubert variety is isotropic.  This gives the following diagram of homotopy groups:
\[
\xymatrix{
  & \pi_2(\op{IGr}(r,F))  \ar[r]\ar[d]_{=} & \pi_1(Z) \ar[r]\ar@{=}[d] & \pi_1(O_{0,r}) \ar[r]\ar[d] & \pi_1(\op{IGr}(r,F))=1 \\
  & \pi_2(\op{Gr}(r,F))  \ar[r] & \pi_1(Z) \ar[r] & \pi_1(O_r)=1 &
}\]
where the left four groups are all $\integ$.
The lower left map is an isomorphism since the cokernel is trivial, thus $\pi(O_{r,0})$ is trivial.

Now consider the second case when $s=r<m$. Then $O_{r,r}$ is also a $G'=\Sp_{2n}\times \SL_m$-orbit. Note that the kernel of a map $\phi \in O_{r,r}$ can be any element in $\op{Gr}(m-r,E)$. This gives a fiber bundle $Z' \to O_{r,r} \to \op{Gr}(m-r,E)$, where $Z'$ can be identified with the maps $\Hom(\complex^r, F)$ of maximal isometry $r$. Now $Z'$ is the open $\GL_r \times \Sp_{2n}$-orbit with connected stabilizer and so $\pi_1(Z')\cong\integ$ by Remark \ref{rem:spherical}. We write $\op{Gr}(m-r,E)$ as a homogeneous space $G' / (\Sp_{2n} \times P)$, where $P$ is the corresponding parabolic subgroup of $\SL_m$ and $\pi_1(P)\cong \integ$. This gives another fiber bundle $P \times \Sp_{2n} \to G' \to \op{Gr}(m-r,E)$ that maps to the previous bundle. We get the following diagram of homotopy groups:
\[
\xymatrix{
& \pi_2(\op{Gr}(m-r,F))  \ar[r]\ar@{=}[d] & \pi_1(P\times \Sp_{2n}) \ar[r]\ar[d] & 1 \ar[r]\ar[d] & 1 \\
1 \ar[r] & \pi_2(\op{Gr}(m-r,F))  \ar[r] & \pi_1(Z') \ar[r] & \pi_1(O_{r,r}) \ar[r] & 1
}\]
In order to show that $O_{r,r}$ is simply connected, we have to show that the map $P\times \Sp_{2n} \to Z'$ induces an isomorphism on the level of fundamental groups. It is easy to see that this map factors through the quotient map $P \times \Sp_{2n} \to \GL_r  \times \Sp_{2n}$ and the quotient map $\GL_r \times \Sp_{2n} \to Z'$, both inducing isomorphism on the level of fundamental groups (the latter since the stabilizer of $Z'$ is connected, as mentioned before).
\end{proof}

In the cases below, all non-open orbits will be either of the form as in Lemma \ref{lem:simply}, or the orbit $O_{3,2}$. In each latter case the orbit $O_{3,2}$ will just be the simply connected variety $O_3$ of matrices of rank $3$. So all the non-open orbits considered below are simply connected. Furthermore, in case we take a representation that is geometrically equivalent to $G\to \GL(X)$, the only possible additional vertices of the quiver of equivariant $\dif$-modules must be isolated vertices (when the semi-invariant $f$ is present, since the $b$-function of $f$ has integer roots) corresponding to integrable connections on the open orbit. Now we proceed describing the quiver for each case.

\medskip
{\bf Case 1:}  $\Sp_{2n}\otimes \GL_2$, with $n\geq 2$.
\medskip

In this case there are $4$ orbits corresponding to $(r,s)\in\{(0,0), (1,0), (2,0), (2,2)\}$ of codimensions $4n, 2n-1, 1, 0,$ in this order. The $b$-function of the semi-invariant $f$ has roots $-1, -2n$. This gives $3$ simple equivariant $\dif$-modules, as in Proposition \ref{prop:root}. We note that the orbit $(1,0)$ is just the orbit of rank $1$ matrices. Using Theorem \ref{thm:gener} and Lemma \ref{lem:local}, we see that there are no arrows between vertices $(0,0)$ and $(2,2)$. The injective hull of the simple at $(2,2)$ is $\CC[X]_f$, and its decomposition shows that there are no arrows between $(2,2)$ and $(1,0)$. Further, the quotient of $\CC[X]_f$ by $\CC[X]$ is the injective hull of the simple at $(2,0)$, and so there are also no arrows between $(2,0)$ and $(1,0)$. We obtain that the quiver of $\op{mod}_G(\dif_X)$ is of type $\AAA_3$

\[\xymatrix{(0,0) \ar@<0.5ex>[r] & \ar@<0.5ex>[l] (2,0)
\ar@<0.5ex>[r] & \ar@<0.5ex>[l] (2,2)}\]
with all 2-cycles zero, and $(1,0)$ is an isolated vertex.

\medskip
{\bf Case 2:}  $\Sp_{2n}\otimes \GL_3$, with $n\geq 2$.
\medskip

Then there are $6$ orbits $(r,s)\in\{(0,0), (1,0), (2,0), (2,2), (3,0), (3,2)\}$ of codimensions $6n, 4n-2, 2n-1, 2n-2, 3, 0,$ respectively (except when $n=2$, when there is no orbit corresponding to $(3,0)$). We note that among all the examples in this paper, this is the only case ($n\geq 3$) when the orbits are not linearly ordered with respect to the inclusion of orbit closures, as $O_{2,2}$ is not in the orbit closure of $O_{3,0}$. (see \cite[Section 13]{howumed}). We will discuss according to two cases: $n=2$, and $n\geq 3$.

\vspace{0.1in}

First, we discuss the case $n=2$ with  $5$ orbits. The complement of the (open) orbit of $(3,2)$ has codimension $\geq 2$, hence it is simply-connected and the pushforward of the simple at $(3,2)$ to $X$ is $\CC[X]$. But this push-forward must be the injective hull of the simple at $(3,2)$ and so $(3,2)$ is has no arrows going  out, and by duality also not in. By the Fourier transform, the same is true for $(0,0)$.

By Lemma \ref{lem:local}, we have an exact sequence (with $U=X \setminus \ove{O}_{2,0}$)
\[ 0 \to H^2_{\ove{O}_{2,2}} (X,\orb_X) \to H_{O_{2,2}}^2(U,\orb_U) \to H_{\ove{O}_{2,0}}^{3}(X,\orb_X) \to  H^{3}_{\ove{O}_{2,2}} (X,\orb_X).\]
Since $\ove{O}_{2,2}$ is just the variety of matrices of rank $<3$, the simple $\dif_X$-module corresponding to $(2,0)$ does not appear as a composition factor in the $\dif_X$-module  $H^{3}_{\ove{O}_{2,2}} (X,\orb_X)$, by the case of generic matrices. This shows that there is an arrow from $(2,0)$ to $(2,2)$. Using the twisted Fourier transform, we obtain that the quiver of $\op{mod}_G(\dif_X)$ is a subquiver of the type $\AAA_3$ quiver
\[\xymatrix{(1,0) \ar@<0.5ex>[r] & \ar@<0.5ex>[l] (2,0)
\ar@<0.5ex>[r] & \ar@<0.5ex>[l] (2,2)}\]
with all 2-cycles zero, and $(3,2)$ and $(0,0)$ are isolated vertices. (Note that the Fourier transform switches $(2,2)$ and $(1,0)$. It must therefore send $(2,0)$ to itself. Hence the arrows between $(2,0)$ and $(1,0)$).  To see that the quiver is $\AAA_3^c$ as in \eqref{eq:doubb}, we show that the composition $(2,2)\to (2,0) \to (1,0)$ is zero (this implies that $(1,0)\to (2,0) \to (2,2)$ is also zero). This follows from the following result.

\begin{lemma}\label{lem:m2}
Let $P_{2,2}$ be the projective cover in $\op{mod}_G(\dif_X)$ of the simple $\dif_X$-module $S_{2,2}$ supported on $\ove{O}_{2,2}$. Then the holonomic length of $P_{2,2}$ is two.
\end{lemma}

\begin{proof}
By the quiver description of $\op{mod}_G(\dif_X)$, the length of $P_{2,2}$ is either two or three. Hence, it is enough to show that the closure of the conormal bundle $\ove{T^*_{O_{1,0}} X}$  is not an irreducible component of the characteristic variety of $P_{2,2}$. We note that $S_{2,2}$ is also a $\GL_4\times \GL_3$-equivariant $\dif_X$-module. By \cite[Section 5]{claudiu1}, the irreducible $\GL_4\times \GL_3$-representation $\bigwedge^2 \CC^4 \otimes \Sym_2 \CC^3$ appears in $S_{2,2}$ with multiplicity one. Restricting $\GL_4$ to $\Sp_4$, we obtain that the $G$-representation $V=\op{triv} \otimes \Sym_2\CC^3$ appears in the $G$-decomposition of $S_{2,2}$. Since $\Sp_4 \otimes \GL_3$ is of Capelli type by \cite[Section 15]{howumed}, we have $P(V)\cong P_{2,2}$ according to Theorem \ref{thm:capelli}. On the other hand, $P(V)$ has an explicit $\dif$-module presentation as described in \eqref{eq:highest}, which is given by elements of the universal enveloping algebra that generate the annihilator of the highest weight vector in $V$. Implementing this presentation in the software Macaulay 2 \cite{M2}, we obtained generators of the corresponding characteristic ideal, whose zeroes form the characteristic variety of $P_{2,2}$. \\

\noindent\textbf{M2 code for characteristic ideal of} $P_{2,2}$:
\begin{verbatim}
load "Dmodules.m2"
R=QQ[x_(1,1)..x_(4,3)];
makeWeylAlgebra R;
g_(1,1)=x_(1,1)*dx_(1,1)+x_(1,2)*dx_(1,2)+x_(1,3)*dx_(1,3);
g_(1,2)=x_(1,1)*dx_(2,1)+x_(1,2)*dx_(2,2)+x_(1,3)*dx_(2,3);
g_(1,3)=x_(1,1)*dx_(3,1)+x_(1,2)*dx_(3,2)+x_(1,3)*dx_(3,3); 
g_(1,4)=x_(1,1)*dx_(4,1)+x_(1,2)*dx_(4,2)+x_(1,3)*dx_(4,3);
g_(2,1)=x_(2,1)*dx_(1,1)+x_(2,2)*dx_(1,2)+x_(2,3)*dx_(1,3);
g_(2,2)=x_(2,1)*dx_(2,1)+x_(2,2)*dx_(2,2)+x_(2,3)*dx_(2,3);
g_(2,3)=x_(2,1)*dx_(3,1)+x_(2,2)*dx_(3,2)+x_(2,3)*dx_(3,3);
g_(2,4)=x_(2,1)*dx_(4,1)+x_(2,2)*dx_(4,2)+x_(2,3)*dx_(4,3);
g_(3,1)=x_(3,1)*dx_(1,1)+x_(3,2)*dx_(1,2)+x_(3,3)*dx_(1,3);
g_(3,2)=x_(3,1)*dx_(2,1)+x_(3,2)*dx_(2,2)+x_(3,3)*dx_(2,3);
g_(3,3)=x_(3,1)*dx_(3,1)+x_(3,2)*dx_(3,2)+x_(3,3)*dx_(3,3);
g_(3,4)=x_(3,1)*dx_(4,1)+x_(3,2)*dx_(4,2)+x_(3,3)*dx_(4,3);
g_(4,1)=x_(4,1)*dx_(1,1)+x_(4,2)*dx_(1,2)+x_(4,3)*dx_(1,3);
g_(4,2)=x_(4,1)*dx_(2,1)+x_(4,2)*dx_(2,2)+x_(4,3)*dx_(2,3);
g_(4,3)=x_(4,1)*dx_(3,1)+x_(4,2)*dx_(3,2)+x_(4,3)*dx_(3,3);
g_(4,4)=x_(4,1)*dx_(4,1)+x_(4,2)*dx_(4,2)+x_(4,3)*dx_(4,3);
h_(1,1)=x_(1,1)*dx_(1,1)+x_(2,1)*dx_(2,1)+x_(3,1)*dx_(3,1)+x_(4,1)*dx_(4,1);
h_(1,2)=x_(1,1)*dx_(1,2)+x_(2,1)*dx_(2,2)+x_(3,1)*dx_(3,2)+x_(4,1)*dx_(4,2);
h_(1,3)=x_(1,1)*dx_(1,3)+x_(2,1)*dx_(2,3)+x_(3,1)*dx_(3,3)+x_(4,1)*dx_(4,3);
h_(2,1)=x_(1,2)*dx_(1,1)+x_(2,2)*dx_(2,1)+x_(3,2)*dx_(3,1)+x_(4,2)*dx_(4,1);
h_(2,2)=x_(1,2)*dx_(1,2)+x_(2,2)*dx_(2,2)+x_(3,2)*dx_(3,2)+x_(4,2)*dx_(4,2);
h_(2,3)=x_(1,2)*dx_(1,3)+x_(2,2)*dx_(2,3)+x_(3,2)*dx_(3,3)+x_(4,2)*dx_(4,3);
h_(3,1)=x_(1,3)*dx_(1,1)+x_(2,3)*dx_(2,1)+x_(3,3)*dx_(3,1)+x_(4,3)*dx_(4,1);
h_(3,2)=x_(1,3)*dx_(1,2)+x_(2,3)*dx_(2,2)+x_(3,3)*dx_(3,2)+x_(4,3)*dx_(4,2);
h_(3,3)=x_(1,3)*dx_(1,3)+x_(2,3)*dx_(2,3)+x_(3,3)*dx_(3,3)+x_(4,3)*dx_(4,3);
I=ideal(g_(1,1)-g_(3,3),g_(1,2)-g_(4,3),g_(2,1)-g_(3,4),g_(2,2)-g_(4,4));
I=I+ideal(g_(1,4)+g_(2,3),g_(1,3),g_(2,4),g_(3,1),g_(3,2)+g_(4,1),g_(4,2));
J=ideal(h_(1,1)+2,h_(2,1),h_(3,1),h_(2,2),h_(3,3));
J=J+ideal(h_(2,3),h_(3,2),h_(1,3)^3,h_(1,2)^3);
K=I+J;
C=charIdeal K;
\end{verbatim}

The code finished in approximately 1.5 hours on a computer with a x86\_64 architecture involving 32 threads, running at 3 GHz. 

Among the generators of the characteristic ideal $C$ obtained, we choose the following element (in $\CC[X\times Y]$ where $Y=X^*$):
\[h(x,y)=x_{2,1} y_{2,3} y_{3,2}-x_{2,1} y_{2,2} y_{3,3}-x_{1,1} y_{2,3}y_{4,2}+x_{1,1}y_{2,2}y_{4,3}-x_{4,1}y_{3,3}y_{4,2}+x_{4,1}y_{3,2}y_{4,3}.\]
We choose the following element $v\in \ove{T^*_{O_{1,0}} X}$:
\[v=\left(\begin{bmatrix} 1 & 0 & 0 \\ 0 & 0 & 0 \\ 0 & 0 & 0 \\ 0 & 0 & 0 \end{bmatrix} , \begin{bmatrix} 0 & 0 & 0 \\ 0 & 1 & 0 \\ 0 & 0 & 0 \\ 0 & 0 & 1 \end{bmatrix} \right).\]
We have $h(v)\neq 0$, finishing the proof of the lemma.
\end{proof}

\vspace{0.1in}

Now let $n\geq 3$ in the proof of the theorem. As before, $(3,2)$ and $(0,0)$ are isolated vertices, and there is an arrow from $(2,0)$ to $(2,2)$. The twisted Fourier transform of 
$(2,2)$ is $(1,0)$. By Corollary \ref{cor:relations} (applied to the variety $\ove{O}_{3,0} \cup \ove{O}_{2,2}$), we see that there are no arrows between $(3,0)$ and $(2,2)$. To understand the arrows at the vertex $(3,0)$ using Lemma \ref{lem:local}, we proceed to describe $H^3_{\ove{O}_{3,0}}(R)$ in more detail, where $R=\orb_X$. Let $S_{r,s}$ denote the simple $\dif$-module with support $\ove{O}_{r,s}$. Then $S_{3,0}$ is a submodule of $H^3_{\ove{O}_{3,0}}(R)$, and their quotient can have as composition factors $S_{2,0}$ and $S_{1,0}$ with multiplicity $\leq 1$.

By Lemma \ref{lem:compint}, there are three polynomials $f_1,f_2,f_3$ whose zeroes give $\ove{O}_{3,0}$, and we have the following exact sequence:
\[R_{f_1 f_2} \oplus R_{f_1 f_3} \oplus R_{f_2 f_3} \to R_{f_1 f_2 f_3} \to H^3_{\ove{O}_{3,0}}(R) \to 0.\]
Let $f:=f_1 f_2 f_3$, and denote by $\hat{f}^{-1}$ the image in $H^3_{\ove{O}_{3,0}}(R)$ of the element $1/f$ in $R_f$. The element $\hat{f}^{-1}$ is a semi-invariant under the action of $\lie = \mathfrak{sp}_{2n}\times \mathfrak{gl}_3$ of (highest) weight $1 \otimes (-2,-2,-2)$ (by convention, we work with the dual action of $G$, so that polynomials get positive weights). Let $S$ be a simple $\dif_X$-module that is a composition factor of $H^3_{\ove{O}_{3,0}}(R)$, and has a $G$-semi-invariant section of weight $1 \otimes (-2,-2,-2)$. Using the equivariant decompositions obtained in \cite[Section 5]{claudiu1}, we see that $S_{1,0}$ has no $\SL_3$-invariant sections. Hence, $S$ is isomorphic to either $S_{3,0}$ or $S_{2,0}$. Then $\tilde{\mathcal{F}}(S)$ is also isomorphic to either $S_{3,0}$ or $S_{2,0}$, and has a $G$-semi-invariant section of weight $1\otimes (-2n+2,-2n+2,-2n+2)$ by \eqref{eq:fourier}. By Lemma \ref{lem:section} below, we must have $\tilde{\mathcal{F}}(S)\cong S_{2,0}$, and $S_{2,0}$ is not a composition factor of $H^3_{\ove{O}_{3,0}}(R)$. This implies that $S_{3,0} \cong S$ is generated by $\hat{f}^{-1}$. Using the twisted Fourier transform, we obtain that the quiver of $\op{mod}_G(\dif_X)$ is a disjoint union of two type $\AAA_2$ quivers
\[\xymatrix{(1,0) \ar@<0.5ex>[r] & \ar@<0.5ex>[l] (3,0)}
\hspace{0.5in} \xymatrix{(2,0) \ar@<0.5ex>[r] & \ar@<0.5ex>[l] (2,2)}
\]
with all 2-cycles zero, and $(3,2)$ and $(0,0)$ are isolated vertices.

\begin{lemma}\label{lem:section}
 The equivariant $\dif_X$-module $H^3_{\ove{O}_{3,0}}(R)$ has no non-zero $G$-semi-invariant sections of (highest) weight $1\otimes (-2n+2,-2n+2,-2n+2)$.
\end{lemma}

\begin{proof}
Although the $\dif_X$-module $R_f$ is not $G$-equivariant, it is $H=\Sp_{2n} \times (\CC^*)^3$-equivariant. By Lemma \ref{lem:compint}, the only $H$-semi-invariant element (up to scalar) of $H$-weight $1\otimes (-2n+2,-2n+2,-2n+2)$ in $R_f$ is $f^{-n+1}$. Hence, if there is any $G$-semi-invariant section of (highest) weight $1\otimes (-2n+2,-2n+2,-2n+2)$ in $H^3_{\ove{O}_{3,0}}(R)$, then it must be the image $\hat{f}^{-n+1}$ of the element $f^{-n+1}$. However, it is obvious that $\hat{f}^{-n+1}$ is not $\mathfrak{sl}_3$-invariant.
\end{proof}

\medskip
{\bf Case 3:}  $\Sp_{4}\otimes \GL_m$, with $m\geq 4$.
\medskip

There are $6$ orbits corresponding to the data $(0,0), (1,0), (2,0), (2,2), (3,2), (4,4)$. These are of codimensions $4m,3m-3, 2m-3, 2m-4, m-3, 0,$ respectively. There is a non-trivial semi-invariant only when $m=4$. Hence we discuss according to these two cases.

When $m>4$, then $(0,0)$ and $(4,4)$ are isolated vertices. By Theorem \ref{thm:gener} and Lemma \ref{lem:local}, $(3,2)$ and $(1,0)$ are also isolated, since the $s$-component of the type $(r,s)$ is here irrelevant once $(2,0)$ is ignored. By Lemma \ref{lem:local} again, there is an arrow from $(2,0)$ to $(2,2)$. Hence the quiver of $\op{mod}_G(\dif_X)$ is the $\AAA_2$ type quiver.

\[\xymatrix{
(2,0) \ar@<0.5ex>[r] & \ar@<0.5ex>[l] (2,2)}\]
with all 2-cycles zero, and $(0,0), (1,0), (3,2), (4,4)$ are isolated vertices.

\vspace{0.1in}

Now let $m=4$. If we delete the vertex $(2,0)$ of the quiver of $\op{mod}_G(\dif_X)$, we obtain a quiver of type $\AAA_5$, by Theorem \ref{thm:gener}. By Lemma \ref{lem:local}, there is an arrow $(2,0)\to (2,2)$, but there are no non-zero paths from $(2,0)$ to $(3,2)$. Using the twisted Fourier transform, we obtain that the quiver of $\op{mod}_G(\dif_X)$ is of type $\EEE$
\[\xymatrix{
& & (2,0) \ar@<0.5ex>[d]^{\alpha} & & \\
(0,0) \ar@<0.5ex>[r] & \ar@<0.5ex>[l] (1,0) \ar@<0.5ex>[r] & \ar@<0.5ex>[u]^{\beta} \ar@<0.5ex>[l] \ar@<0.5ex>[r] (2,2) & \ar@<0.5ex>[l] (3,2)  \ar@<0.5ex>[r] & \ar@<0.5ex>[l] (4,4)
}\]
with all 2-cycles zero, and all compositions with the arrows $\alpha$ or $\beta$ equal to zero.

\subsection{The other cases}

Lastly, we discuss the cases $ \Sp_n\,, \spin_{10}\, , \SO_n \otimes \CC^*$ for $n\geq 3\,$, $\spin_7 \otimes \CC^*\,$, $\spin_9 \otimes \CC^*\,,\Gtwo \otimes \CC^*\,$, $\E\otimes \CC^*$ in this order. Since the techniques are analogous to the previous cases, we skip the details. For the orbit structure of these spaces, cf. \cite[Section 13]{howumed}.

\medskip
{\bf Case 4:}  $\Sp_{2n}$
\medskip

Here the group $G=\Sp_n$ acts on its fundamental representation $X$ that is $2n$-dimensional. There is only one non-zero orbit, whose complement has codimension $\geq 2$ in $X$. Hence, the orbits are simply connected, and the quiver of $\op{mod}_G(\dif_X)=\op{mod}_{G\times \complex^*}(\dif_X)=\op{mod}_\Lambda^{rh}(\dif_X)$ consists of $2$ isolated vertices.

\medskip
{\bf Case 5:} $\spin_{10}$
\medskip

Here $G=\spin_{10}$ and $X$ is the even half-spin representation. There are only 2 non-zero orbits. The open orbit has complement with codimension $\geq 2$, and the highest weight orbit has connected $G$-stabilizer by Lemma \ref{lem:highest}. Hence, all orbits are simply connected, so that the quiver of $\op{mod}_G(\dif_X)=\op{mod}_{G\times \complex^*}(\dif_X)=\op{mod}_\Lambda^{rh}(\dif_X)$ consists of $3$ isolated vertices. (Since there is no semi-invariant, the simple on the big orbit is isolated, and by duality so is the simple on the point. Hence, so must be the last. Compare the argument in the case $\Sp_{2n}\otimes \GL_3$ above).

\medskip
{\bf Case 6:} $\SO_n \otimes \CC^*$ for $n\geq 3$
\medskip

Here $G=\SO_n \times \CC^*$ acts on the fundamental representation $X$ of dimension $n$. There are 2 non-zero orbits, the open orbit $O_2$, and $O_1$ whose closure is a hypersurface. The quadratic semi-invariant $f$ has degree $2$, and the roots of its $b$-function are $-1$, $-n/2$ (see \cite[Section 12]{kimu} or \cite[Example 2.9]{bub2}). As seen in Proposition \ref{prop:semiorbit}, $O_1$ is simply connected. By Remark \ref{rem:spherical}, there are (up to isomorphism) two simple equivariant $\dif$-modules with full support. Using our previous methods, we obtain the following result.

\begin{theorem}\label{thm:ortho} 
Let $G=\SO_n \times \CC^*$ and $X$ as above. Then  $\op{mod}_G (\dif_X) \cong \rep(\widehat{Q})$, where the quiver $\widehat{Q}$ has two connected components as follows.
\begin{itemize}
\item[(a)] if $n$ is even, then $\widehat{Q}$ is of type $\AAA_3$
\[\xymatrix{(0) \ar@<0.5ex>[r] & \ar@<0.5ex>[l] (1)
\ar@<0.5ex>[r] & \ar@<0.5ex>[l] (2)}\]
with all 2-cycles zero, and $(2')$ is an isolated vertex.
\item[(b)] if $n$ is odd, then $\widehat{Q}$ is
\[\xymatrix{(1) \ar@<0.5ex>[r] & \ar@<0.5ex>[l] (2)}
\hspace{0.5in} \xymatrix{(0) \ar@<0.5ex>[r] & \ar@<0.5ex>[l] (2')}
\]
with all 2-cycles zero.
\end{itemize}
\end{theorem}

\bigskip
{\bf Case 7:} $\spin_7 \otimes \CC^*$
\medskip

Here $G=\spin_7 \times \CC^*$ acting on the $8$-dimensional spin representation $X$. The image of the action map is contained in that of $\SO_8 \otimes \CC^*$ (see \cite[Section 7]{saki}), and it has the same $3$ orbits. The quiver of $\op{mod}_G (\dif_X)$ is the same as in Theorem \ref{thm:ortho} (a).

\medskip
{\bf Case 8:} $\spin_9 \otimes \CC^*$
\medskip

Here $G=\spin_9 \times \CC^*$ is acting on the $16$-dimensional spin representation $X$. The image of the action map is contained in that of $\SO_{16} \otimes \CC^*$ (see \cite{saki}). However, there are four $G$-orbits $O_0,O_1,O_2,O_3$, as the highest weight $\SO_{16}$-orbit decomposes as the union of $O_2$ and $O_1$ of codimensions $1$ and $5$, respectively (see \cite[Section 13]{howumed}). Hence, $O_2$ is also simply connected, and by Lemma \ref{lem:highest}, the same is true for $O_1$. Using \cite[Lemma 2.4]{binary} one can see that $(1)$ is an isolated vertex, since $H^1_{\ove{O}_2}(X,\orb_X)=\complex[X]_f/\complex[X]$ is the injective cover in $\op{mod}_G^{\ove{O}_2}(\dif_X)$ of the simple supported on $\ove{O}_2$. Using Theorem \ref{thm:ortho} we obtain that the quiver of $\op{mod}_G(\dif_X)$ is of type $\AAA_3$
\[\xymatrix{(0) \ar@<0.5ex>[r] & \ar@<0.5ex>[l] (2)
\ar@<0.5ex>[r] & \ar@<0.5ex>[l] (3)}\]
with all 2-cycles zero, and $(1)$ and $(3')$ are isolated vertices.

\medskip
{\bf Case 9:} $\Gtwo \otimes \CC^*$
\medskip

 Here  $G=\Gtwo \times \CC^*$ acts on a $7$-dimensional space $X$. The image of the action map is contained in that of $\SO_7\otimes \CC^*$, and it has the same $3$ orbits. The quiver of $\op{mod}_G (\dif_X)$ is the same as in Theorem \ref{thm:ortho} (b). 

\medskip
{\bf Case 10:} $\E\otimes \CC^*$
\medskip

Here $G=\E\times \CC^*$ acts on its $27$-dimensional fundamental representation $X$. There are $4$ orbits $O_0,O_1,O_2,O_3$, and there is a semi-invariant $f$ of degree $3$. The roots of the $b$-function of $f$ are $-1,-5,-9$ (see \cite[Section 6]{kimu}). Since the roots are integers, the simples supported on $X$ that are not isomorphic to $\CC[X]$ must be projective-injective. By Lemma \ref{lem:highest}, the highest weight orbit $O_1$ is simply connected. Hence, up to isolated vertices, the quiver of $\op{mod}_G(\dif_X)$ is of type $\AAA_4$
\[\xymatrix{(0) \ar@<0.5ex>[r] & \ar@<0.5ex>[l] (1)
\ar@<0.5ex>[r] & \ar@<0.5ex>[l] (2)
\ar@<0.5ex>[r] & \ar@<0.5ex>[l] (3)}\]
with all 2-cycles zero. Note that even if the stabilizer of $O_2$ is not connected, the corresponding additional vertices must be isolated because they can't have arrows to (or from) vertex (1), since the characteristic cycle of the projective cover of the simple at $(1)$ is multiplicity-free by Corollary \ref{cor:projmult}.

\section{Concluding remarks}

\subsection{Explicit presentations of $\dif$-modules}

Let us assume, for simplicity, that $X$ is an affine space of Capelli type (see Definition \ref{def:capelli}). By Theorem \ref{thm:capelli}, each indecomposable projective in $\op{mod}_G(\dif_X)$ is isomorphic to $P(\lambda)$, for some $\lambda \in \Pi$. These $\dif$-modules have explicit presentations as described in \eqref{eq:highest}, which can be useful in explicit computations (see proof of Lemma \ref{lem:m2}). Furthermore, since any equivariant $\dif$-module has a projective resolution in $\op{mod}_G(\dif_X)$, one can in principle give explicit presentations to arbitrary $\dif$-modules in $\op{mod}_G(\dif_X)$.

For this, one needs to find for a simple $\dif_X$-module $S$ a weight $\lambda \in \Pi$ such that $m_\lambda(S) \neq 0$ (in which case it is equal to $1$). There are several approaches, for example see \cite{claudiu1} or Proposition \ref{prop:root}. We consider the following example.

Let $X$ be the space of $m\times n$ matrices with $G=\GL_m \times \GL_n$ as in Section \ref{subsec:gener}. When $m\neq n$, the category $\op{mod}_G(\dif_X)$ is semi-simple, hence the indecomposable projective $\dif$-modules are in fact simple. Using the $G$-equivariant decompositions of the simples as described in \cite[Section 5]{claudiu1}, we obtain explicit presentations as mentioned above.

When $m=n$, each simple in $\op{mod}_G(\dif_X)$ has a $G$-semi-invariant section, as seen in Section \ref{subsec:gener}. More precisely, let $\sigma$ denote the character $\det:G \to \CC$, and $S_i$ the simple supported on the subset $\ove{O}_i$ of matrices of rank $\leq i$, for $i=0,\dots, n$. Then we have $m_{\sigma^{i-n}}(S_i) =1$, hence the projective cover of $S_i$ is $P_i:=P(\sigma^{i-n})$. Using the quiver-theoretic description as in Theorem \ref{thm:gener}, we obtain easily the projective resolutions of the simples $S_i$. For example, for $0< i < n$ we have the periodic resolution of $S_i$
\[\dots \rightarrow P_i^2 \rightarrow P_{i-1}\oplus P_{i+1} \rightarrow P_i^2 \rightarrow  P_{i-1}\oplus P_{i+1} \rightarrow P_i \rightarrow S_i \rightarrow 0.\]
To fully understand the presentation of $S_i$, we describe the maps between the projectives. Let $f$ denote the determinant in the generic variables and $f^*$ the determinant in the partial variables. Take two projectives $P_i, P_j$, with $i\leq j$. We know that $\Hom_{\dif_X}(P_i,P_j) = \Hom_{\dif_X}(P_j,P_i)=\CC$. The respective non-zero maps can be described on the level of the generators as $1\otimes 1 \mapsto f^{j-i}\otimes 1$ and $1\otimes 1 \mapsto (f^*)^{j-i}\otimes 1$. Moreover, we can realize the equivalence of categories $\op{mod}_G(\dif_X) \cong \rep(\AAA_{n+1})$ from Theorem \ref{thm:gener} explicitly by sending an object $M\in \op{mod}_G(\dif_X)$ to the following representation of $\rep(\AAA_{n+1})$:
\[\xymatrix{
M_{\sigma^{-n}} \ar@<0.5ex>[r]^{f} & \ar@<0.5ex>[l]^{f^*} M_{\sigma^{-n+1}}
\ar@<0.5ex>[r]^-{f} & \ar@<0.5ex>[l]^-{f^*} \dots
\ar@<0.5ex>[r]^{f} &\ar@<0.5ex>[l]^{f^*} M_{\sigma^{-1}}
\ar@<0.5ex>[r]^-{f} & \ar@<0.5ex>[l]^-{f^*} M_{\sigma^0}}.
\]
Here, $M_{\sigma^i}$ denotes the weight space of $M$ corresponding to $\sigma^i$, and $M_{\sigma^i} \xrightarrow{f} M_{\sigma^{i+1}}$ (resp.\ $M_{\sigma^{i+1}} \xrightarrow{f^*} M_{\sigma^{i}}$) denotes the linear map induced by multiplication by $f$ (resp.\ $f^*$). This realization appears essentially in \cite[Section 7]{nang}.

\subsection{Characteristic cycles and the Pyasetskii pairing}

In all cases from Section \ref{sec:examples}, we can describe the Pyasetskii pairing using the Fourier transform as explained in Section \ref{subsec:four}. Moreover, for almost all simple equivariant $\dif$-modules we can describe their corresponding characteristic cycles; note that by Corollary \ref{cor:projmult}, this follows from the description of the characteristic varieties only. 

Let us exhibit these considerations in the case $\Sp_{2n} \otimes \GL_3$, with $n\geq 3$. As seen in Section \ref{subsec:sp}, the quiver of $\op{mod}_G(\dif_X)$ is 
\[\xymatrix{(1,0) \ar@<0.5ex>[r] & \ar@<0.5ex>[l] (3,0)}
\hspace{0.5in} \xymatrix{(2,0) \ar@<0.5ex>[r] & \ar@<0.5ex>[l] (2,2)}
\]
with all 2-cycles zero, and $(3,2)$ and $(0,0)$ are isolated vertices. Let $S_{r,s}$ denote the simple supported on $\ove{O}_{r,s}$. We have \[\tilde{\mathcal{F}}(S_{3,2})=S_{0,0}\,, \,\tilde{\mathcal{F}}(S_{2,2})= S_{1,0} \,, \, \tilde{\mathcal{F}}(S_{3,0})=S_{2,0}.\]
The Pyasetskii pairing is given by (see \eqref{eq:four-char})
\[ O_{3,2}^{\vee} = O_{0,0} \, , \, O_{2,2}^{\vee} = O_{1,0} \, , \, O_{3,0}^{\vee} = O_{2,0}.\]
This shows that in this case the characteristic varieties of the simples are irreducible and 
\[\charC (S_{r,s}) = [\ove{T^*_{O_{r,s}} X}]\,,\, \mbox{ for all } (r,s).\]

\section*{Acknowledgments}
We would like to thank Claudiu Raicu and Jerzy Weyman for fruitful conversations and helpful commentary.

\bibliography{biblo}
\bibliographystyle{amsplain}

\end{document}